\documentclass[a4paper, reqno]{amsart}

\usepackage[utf8]{inputenc}
\usepackage[usenames,dvipsnames]{color}
\usepackage{amsthm,amsfonts,amssymb,amsmath,amsxtra}
\usepackage[all]{xy}
\SelectTips{cm}{}
\usepackage{xr-hyper}
\usepackage[colorlinks=true,
citecolor=black,
linkcolor=red,
urlcolor=blue]{hyperref}
\usepackage{verbatim}
\usepackage{caption}

\usepackage[margin=1.25in]{geometry}
\usepackage{mathrsfs}

\usepackage[symbol]{footmisc}

\usepackage{tabularx}
\usepackage{makecell}

\usepackage{booktabs}

\usepackage{tikz}

\makeatletter
\newcommand*\circled[2][1.6]{\tikz[baseline=(char.base)]{
    \node[shape=circle, draw, inner sep=1pt,
        minimum height={\f@size*#1},] (char) {\vphantom{WAH1g}#2};}}
\makeatother

\RequirePackage{xspace}
\RequirePackage{etoolbox}
\RequirePackage{varwidth}
\RequirePackage{enumitem}
\RequirePackage{tensor}
\RequirePackage{mathtools}
\RequirePackage{longtable}
\RequirePackage{multirow}

\newcommand{\Id}{\textrm{id}}

\def\i{^{-1}}
\def\ge{\geqslant}
\def\le{\leqslant}
\def\<{\langle}
\def\>{\rangle}

\def\ba{\textbf{a}}

\def\Adm{{\rm{Adm}}}
\def\ad{{\rm{ad}}}

\def\a{\alpha}
\def\b{\beta}
\def\g{\gamma}
\def\G{\Gamma}
\def\d{\delta}

\def\o{\omega}

\def\s{\sigma}
\def\t{\tau}
\def\th{\theta}
\def\k{\kappa}
\def\l{\lambda}
\def\z{\zeta}

\def\tPhi{\tilde \Phi}

\def\ZZ{\mathbb Z}

\def\GG{\mathbb G}
\def\NN{\mathbb N}
\def\QQ{\mathbb Q}
\def\JJ{\mathbb J}
\def\FF{\mathbb F}
\def\RR{\mathbb R}

\def\PP{\mathbb P}
\def\SS{\mathbb S}
\def\kk{\bold{k}}

\def\ca{\mathcal A}

\def\co{\mathcal O}
\def\cp{\mathcal P}

\def\car{\mathcal R}
\def\cs{\mathcal S}

\def\car{\mathcal R}

\def\tta{{\tilde \alpha}}

\def\tW{\tilde W}
\def\tw{{\tilde w}}

\def\ba{\bold{a}}
\def\Gal{{\rm Gal}}
\def\ad{{\rm ad}}

\def\brF{{\breve F}}
\def\Res{{\rm Res}}
\def\GL{{\rm GL}}

\def\pr{{\rm pr}}
\def\der{{\rm der}}
\def\distg{{\rm dist}}

\theoremstyle{plain}
\newtheorem{thm}{Theorem}[section]
\newtheorem{conj}{Conjecture}[section]
\newtheorem{hyp}{Hypothesis}[section]
\newtheorem*{thm*}{Theorem}
 \newtheorem{prop}[thm]{Proposition}
 \newtheorem{lem}[thm]{Lemma}
 \newtheorem{cor}[thm]{Corollary}

\theoremstyle{definition}

\theoremstyle{remark}

\newtheorem{rmk}[thm]{Remark}
\newtheorem*{claim*}{Claim}

\begin{document}

\title{Connectedness of affine Deligne-Lusztig varieties for unramified groups}
\author{Sian Nie}
\address{Institute of Mathematics, Academy of Mathematics and Systems Science, Chinese Academy of Sciences, 100190, Beijing, China}
\email{niesian@amss.ac.cn}

\begin{abstract}
For unramified reductive groups, we determine the connected components of affine Deligne-Lusztig varieties in the partial affine flag varieties. Based on the work of Hamacher-Kim and Zhou, this result allows us to verify, in the unramified group case, the He-Rapoport axioms, the ``almost product structure" of Newton strata, and the precise description of mod $p$ isogeny classes predicted by the Langlands-Rapoport conjecture, for the Kisin-Pappas integral models of Shimura varieties of Hodge type with parahoric level structure.
\end{abstract}

\maketitle

\section*{Introduction}
\subsection{} \label{back} Let $F$ be a non-Archimedean local field with valuation ring $\co_F$ and residue field $\FF_q$, where $q$ is a power of some prime $p$. Denote by $\brF$ the completion of a maximal unramified extension of $F$. Let $G$ be a connected reductive group defined over $F$, and let $\s$ be the Frobenius automorphism of $G(\brF)$. Fix an element $b \in G(\brF)$, a geometric cocharacter $\l$ of $G$, and a $\s$-stable parahoric subgroup $K \subseteq G(\brF)$. The attached affine Deligne-Lusztig variety is defined by $$X(\l, b)_K = X^G(\l, b)_K = \{g \in G(\brF) / K; g\i b \s(g) \in K \Adm(\l) K \},$$ where $\Adm(\l)$ is the admissible set associated to the geometric conjugacy class of $\l$. If $F$ is of equal characteristic, $X(\l, b)_K$ is a locally closed and locally finite-type subvariety of the partial affine flag variety $G(\brF) / K$. If $F$ is of mixed characteristic, $X(\l, b)_K$ is a perfect subscheme of the Witt vector partial affine flag variety, in the sense of Bhatt-Scholze \cite{BS} and Zhu \cite{Zhu}.

The variety $X(\l, b)_K$, first introduced by Rapoport \cite{R}, encodes important arithmetic information of Shimura varieties. Let $(\mathbf{G}, X)$ be a Shimura datum with $G = \mathbf{G}_{\QQ_p}$ and $\l$ the inverse of the Hodge cocharacter. Suppose there is a good integral model for the corresponding Shimura variety with parahoric level structure. Langlands \cite{Lan}, and latter refined by Langlands-Rapoport \cite{LR} and Rapoport \cite{R}, conjectured a precise description of $\overline \FF_p$-points of the integral model in terms of the varieties $X(\l, b)_K$. In the case of PEL Shimura varieties, $X(\l, b)_K$ is also the set of $\overline \FF_p$-points of a moduli space of $p$-divisible groups define by Rapoport-Zink \cite{RZ}.

\subsection{}
The main purpose of this paper is to study the set $\pi_0(X(\l, b)_K)$ of connected components of $X(\l, b)_K$. Notice that $X(\l, b)_K$ only depends on $\l$ and the $\s$-conjugacy class $[b]$ of $b$. Thanks to He \cite{He2}, $X(\l, b)_K$ is non-empty if and only if $[b]$ belongs to the set of ``neutral acceptable" $\s$-conjugacy classes of $G(\brF)$ with respect to $\l$.

Let $\pi_1(G)_{\G_0}$ be the set of coinvariants of the fundamental group $\pi_1(G)$ under the Galois group $\G_0 = \Gal(\overline{\brF} / \brF)$. There is a natural map $\eta_G : G(\brF) / K \to \pi_1(G)_{\G_0}$. To compute $\pi_0(X(\l, b)_K)$ we can assume that $G$ is adjoint and hence simple by the following Cartesian diagram (see \cite[Corollary 4.4]{HZ})
\[
\xymatrix{
  \pi_0(X^G(\l, b)_K) \ar[d]_{\eta_G} \ar[r] & \pi_0(X^{G_{\ad}}(\l_{\ad}, b_{\ad})_{K_{\ad}}) \ar[d]^{\eta_{G_{\ad}}} \\
  \pi_1(G)_{\G_0} \ar[r] & \pi_1(G_{\ad})_{\G_0}.  }
\]
The map $\eta_G$ gives a natural obstruction to the connectedness of $X(\l, b)_K$. Another more technical obstruction is given by the following Hodge-Newton decomposition theorem.
\begin{thm} [{\cite[Theorem 4.17]{GHN2}}] \label{HN-dec}
Suppose $G$ is adjoint and simple. If the pair $(\l, b)$ is Hodge-Newton decomposable (with respect to some proper Levi subgroup $M$) in the sense of \cite[\S 2.5.5]{GHN2}, then $X(\l, b)_K$ is a disjoint union of open and closed subsets, which are isomorphic to affine Deligne-Lusztig varieties attached to $M$.
\end{thm}
By Theorem \ref{HN-dec} and induction on the dimension of $G$, it suffices to consider the Hodge-Newton indecomposable case. This means that either $\l$ is a central cocharacter or the pair $(\l, b)$ Hodge-Newton irreducible, see \cite[Lemma 5.3]{Zh}. In the former case, $$X(\l, b)_K \cong \JJ_b / (K \cap \JJ_b)$$ is a discrete subset with $\JJ_b$ the $\s$-centralizer of $b$. In the latter case, we have the following conjecture.
\begin{conj} [{see \cite[Conjecture 5.4]{Zh}}] \label{main}
Assume $G$ is adjoint and simple. If $(\l, b)$ is Hodge-Newton irreducible, then the map $\eta_G$ induces a bijection $$\pi_0(X(\l, b)_K) \cong \pi_1(G)_{\Gamma_0}^\s,$$ where $\pi_1(G)_{\Gamma_0}^\s$ is the set of $\s$-fixed point of $\pi_1(G)_{\Gamma_0}$.
\end{conj}
If $G$ is unramified (that is, $G$ extends to a reductive group over $\co_\brF$) and $K$ is hyperspecial, Conjecture \ref{main} is established by Viehmann \cite{Vie}, Chen-Kisin-Viehmann \cite{CKV}, and the author \cite{N1}. If $b$ is basic, it is proved by He-Zhou \cite{HZ}. If $G$ is split or $G = \Res_{E  /F} \GL_n$ with $E / F$ a finite unramified field extension, it is proved by L. Chen and the author in \cite{CN} and \cite{CN1}.

The main result of this paper is the following.
\begin{thm} \label{intro}
Conjecture \ref{main} is true if $G$ is unramified.
\end{thm}
In particular, Theorem \ref{intro} completes the computation of connected components of affine Deligne-Lusztig varieties for unramified groups.

\subsection{} We discuss some applications. Assume $p \neq 2$. Let $(\mathbf{G}, X)$ be a Shimura datum of Hodge type with parahoric level structure such that $p \nmid |\pi_1(\mathbf{G}_\der)|$, $\mathbf{G}_{\QQ_p}$ is tamely ramified, and the parahoric subgroup $K$ at $p$ is a connected parahoric. Let $\mathscr{S}_K = \mathscr{S}_K(G, X)$ be the Kisin-Pappas integral model of the corresponding Shimura variety constructed in \cite{KP}. Let $F = \QQ_p$, $G = \mathbf{G}_{\QQ_p}$, and $\l$ be the inverse of the Hodge cocharacter.

\begin{rmk}
In \cite{PR}, Pappas and Rapoport obtained a new construction of integral models for Hodge type Shimura varieties with parahoric level structure, without the the tameness assumption on $G$. It would be desirable to extend the applications discussed below to their integral models in the unramified group case.
\end{rmk}

\subsubsection{} A major motivation to study $\pi_0(X(\l, b)_K)$ comes from the Langlands-Rapoport conjecture mentioned in \S\ref{back}. In the hyperspecial level structure case, the conjecture is proved by Kottwitz \cite{Kott} for PEL Shimura varieties of types $A$ and $C$, and by Kisin \cite{Ki} for his integral models \cite{Ki0} of Shimura varieties of abelian type. Using the Kisin-Pappas integral models \cite{KP} for Hodge type Shimura varieties with parahoric level structure $K$, Zhou \cite{Zh} proved that each mod $p$ isogeny class has the predicted form when $G$ is residually split.

One of the key ingredients in the proofs of Kisin and Zhou is to construct certain lifting map from $X(\l, b)_K$ to an isogeny class of $\mathscr{S}_K (\overline \FF_p)$ (see also \cite[Axiom A]{HK}), which uses in a crucial way descriptions of $\pi_0(X(\l, b)_K)$ in \cite{CKV} and \cite{HZ} respectively. Combining \cite[Proposition 6.5]{Zh} with Theorem \ref{intro}, we deduce that such a lifting map always exists if $G$ is unramified.
\begin{prop} \label{lift}
If $G$ is unramified, then the Rapoport-Zink uniformisation map admits a unique lift on $\overline \FF_p$-points $$ X(\l, b)_K \to \mathscr{S}_K (\overline \FF_p),$$ which respects canonical crystalline Tate tensors on both sides.
\end{prop}
If $G$ is unramified and $K$ is hyperspecial, Proposition \ref{lift} is proved by Kisin \cite{Ki}. If $b$ is basic or $G$ is residually split, it is proved by Zhou \cite{Zh}. If $G$ is quasi-split and $K$ is absolutely special, it is proved by Zhou in \cite[Theorem A.4.3]{Ho}.

As an application, one can extend \cite[Theorem 1.1]{Zh} to the unramified group case, by combining the methods in \cite{Zh} and Proposition \ref{lift}. This is pointed out to us by Zhou.
\begin{cor} \label{LR-conj}
If $G$ is unramified, then the isogeny classes in $\mathscr{S}_K (\overline \FF_p)$ has the form predicted by the Langlands-Rapoport conjecture. Moreover, each isogeny class contains a point which lifts to a special point in the corresponding Shimura variety.
\end{cor}
Corollary \ref{LR-conj} was first proved by Pol van Hoften \cite{Ho} using a different approach.

\subsubsection{} In \cite{HR}, He and Rapoport formulated five axioms on Shimura varieties with parahoric level structure, which provide a group-theorectic way to study certain characteristic subsets (such as Newton strata, Ekedahl-Oort strata, Kottwitz-Rapoport strata, and so on) in the mod $p$ reductions of Shimura varieties. Based on this axiomatic approach, Zhou \cite{Zh} proved that all the expected Newton strata are non-empty (see \cite{KMS} using a different approach). For more applications of these axioms, we refer to \cite{HR}, \cite{HN1}, \cite{GHN2}, \cite{Zh} and \cite{SYZ}. Combining \cite[Theorem 8.1]{Zh} with Proposition \ref{lift} we have
\begin{cor}
All the He-Rapoport axioms hold if $G$ is unramified.
\end{cor}
These axioms are verified by He-Rapoport \cite{HR} in the Siegel case, and by He-Zhou \cite{HZ} for certain PEL Shimura varieties (unramified of types $A$ and $C$ and odd ramified unitary groups). In \cite{Zh}, Zhou verified all the axioms except the surjectivity of \cite[Axiom 4 (c)]{HR} in the general case, and all of them in the case that $G$ is residually split.

\subsubsection{} In \cite{Man}, Mantovan established a formula expressing the $l$-adic cohomology of proper PEL Shimura varieties in terms of the $l$-adic cohomology with compact supports of the Igusa varieties and of the Rapoport-Zink spaces for any prime $l \neq p$. This formula encodes nicely the local-global compatibility of the Langlands correspondence. A key part of its proof is to show that the products of reduced fibers of Igusa varieties and Rapoport-Zink spaces form nice ``pro-\'{e}tale covers up to perfection" for the Newton strata, of PEL Shimura varieties with hyperspecial level structure. This is referred as the ``almost product structure" of Newton strata. In \cite{HK}, Hamacher-Kim extended Mantovan's results to the Kisin-Pappas integral models under some mild assumptions. Combining \cite[Theorem 2]{HK} with Proposition \ref{lift} we have
\begin{cor}
The ``almost product structure" of Newton strata holds if $G$ is unramified.
\end{cor}
When $G$ is unramified and $K$ is hyperspecial, the ``almost product structure" of Newton strata is established by Mantovan \cite{Man} for PEL Shimura varieties. The general case is proved by Hamacher-Kim provided the lifting property \cite[Axiom A]{HK} holds.

\subsection{} We outline the strategy of the proof. First we show the $\s$-centralizer $\JJ_b$ acts transitively on $\pi_0(X(\l, b)_K)$. Then we show the stabilizer of each connected component is the normal subgroup $\JJ_b \cap \ker(\eta_G)$. Combining these two results one deduces that $\pi_0(X(\l, b)_K) \cong \JJ_b / (\JJ_b \cap \ker(\eta_G)) \cong \pi_1(G)^\s$ as desired.

The stabilizers can be determined by adapting the computations in \cite{N2}. The crucial part is to the show the transitivity of the $\JJ_b$ action. Our starting point is the following natural surjection (see Theorem \ref{str}) $$\sqcup_{\tw \in \cs_{\l, b}} \JJ_{b, \tw} \twoheadrightarrow \pi_0(X(\l, b)_K),$$ where $\cs_{\l, b}$ is the set of semi-standard elements (see \S \ref{ss}) contained in $\Adm(\l)$ and $[b]$, and $\JJ_{b, \tw} = \{g \in G(\brF); g\i b \s(g) = \tw\}$ on which $\JJ_b$ acts transitively. So it remains to connect all the subsets $\JJ_{b, \tw} K/K$ in $X(\l, b)_K$. To this end we consider the following decomposition $$\cs_{\l, b} = \sqcup_{x \in \cs_{\l, b}^+} \cs_{\l, b, x},$$ where $\cs_{\l, b}^+$ consists of standard elements in $\cs_{\l, b}$, and $\cs_{\l, b, x}$ consists of elements in $\cs_{\l, b}$ that are $\s$-conjugate to $x \in \cs_{\l, b}^+$ under the Weyl group of $G$. Note that $\cs_{\l, b}^+$ can be naturally identified with a subset of cocharacters dominated by $\l$. So we can adapt the connecting algorithm in \cite{N2} to connect $\JJ_{b, x} K/K$ for $x \in \cs_{\l, b}^+$ with each other. In contrast, the structure of $\cs_{\l, b, x}$ is much more mysterious, which makes it an essential difficulty to connect $\JJ_{b, \tw} K/K$ for $\tw \in \cs_{\l, b, x}$ with each other. To overcome it, we show that each set $\cs_{\l, b, x}$ contains a unique (distinguished) element $x_\distg$ which is of minimal length in its Weyl group coset, and then connect $\JJ_{b, \tw} K/K$ with $\JJ_{b, x_\distg} K/K$ for all $\tw \in \cs_{\l, b, x}$. This new connecting algorithm is motivated from the partial conjugation method by He in \cite{He1} and \cite{He0}.

\subsection{} The paper is organized as follows. In \S \ref{pre} we recall some basic notions and introduce the semi-standard elements. In \S \ref{sec-main} we outline the proof of the main result. In \S \ref{sec-P} we introduce the set $\cp_\tw$ which will play an essential role in our new connecting algorithm. In \S \ref{sec-leaf}, we introduce the new connecting algorithm and use it to connect $\JJ_{b, \tw} K/K$ for $\tw \in \cs_{\l, b, x}$ with each other. In \S \ref{sec-surj} we connect $\JJ_{b, x} K/K$ for $x \in \cs_{\l, b}^+$ with each other. In \S \ref{sec-normal}, \S \ref{sec-ker} and \S \ref{sec-order-3} we compute the stabilizer in $\JJ_b$ of each connected component of $X(\l, b)_K$.

\subsection*{Acknowledgment} We are grateful to Xuhua He for his detailed comments which greatly improve the exposition of the paper. We would like to thank Michael Rapoport for bringing his joint work \cite{PR} with Georgios Pappas to our attention. We are also grateful to Rong Zhou for pointing out the application (Corollary \ref{LR-conj}) to the Langlans-Rapoport conjecture.

\section{Preliminaries} \label{pre}
In the body of the paper we assume that $G$ is unramified, simple, and adjoint. Without loss of generality, we assume further that $F = \FF_q((t))$. Then $\brF = \kk ((t))$ with valuation ring $\co_\brF = \kk[[t]]$ and residue field $\kk = \overline \FF_q$.

\subsection{} \label{setup} Let $T \subseteq B$ be a maximal torus and a Borel subgroup defined over $\co_F$. Let $\car=(Y, \Phi^\vee, X, \Phi, \SS_0)$ be the root datum associated to the triple $(T \subseteq B \subseteq F)$, where $X$ and $Y$ are the character and cocharacter groups of $T$ respectively equipped with a perfect pairing $\< , \>: Y \times X \to \ZZ$; $\Phi = \Phi_G \subseteq X$ (resp. $\Phi^\vee \subseteq Y$) is the set of roots (resp. coroots); $\SS_0$ is the set of simple roots appearing in $B$. For $\a \in \Phi$, we denote by $s_\a$ the reflection which sends $\mu \in Y$ to $\mu-\<\mu, \a\> \a^\vee$, where $\a^\vee \in \Phi^\vee$ denotes the coroot of $\a$. Via the bijection $\a \leftrightarrow s_\a$, we also denote by $\SS_0$ the set of simple reflections.

Let $W_0 = N_T(\brF) / T(\brF)$ be the Weyl group of $G$, where $N_T$ is the normalizer of $T$ in $G$. The Iwahori-Weyl group of $G$ is given by $$\tW_G = \tW = N_T(\brF) / T(\co_\brF) = Y \rtimes W_0=\{t^\mu w; \mu \in Y, w \in W_0\}.$$ We can view $\tW$ as a subgroup of affine transformations of $V := Y \otimes_\ZZ \RR$, where the action of $\tw=t^\mu w$ is given by $v \mapsto \mu+w(v)$. Let $\Phi^+ = \Phi \cap \ZZ_{\ge 0} \SS_0$ be the set of positive roots and let $\ba=\{v \in Y_\RR; 0 < \<\a, v\> < 1, \a \in \Phi^+\}$ be the base alcove.

Let $\tPhi = \tPhi_G = \Phi \times \ZZ$ be the set of (real) affine roots. Let $\tta = \a + k \in \tPhi$. Then $\tilde \a$ is an affine function on $V$ such that $\tilde \a(v) = -\<\a, v\> + k$. The induced action of $\tW$ on $\tPhi$ is given by $\tw(\tilde \a)(v)=\tilde \a(\tw\i(v))$. Let $s_{\tilde \a} = t^{k \a^\vee} s_\a \in \tW$ be the corresponding affine reflection. Then $\{s_{\tilde \a}; \tilde \a \in \tPhi\}$ generates the affine Weyl group $$W^a = W^a_G = \ZZ \Phi^\vee \rtimes W_0=\{t^\mu w; \mu \in \ZZ \Phi^\vee, w \in W_0\}.$$ Moreover, we have $\tW = W^a \rtimes \Omega$, where $\Omega = \Omega_G = \{\o \in \tW; \o(\ba)=\ba\}$. Set $\tPhi^+ = \tPhi_G^+ = \{\tilde \a \in \tPhi; \tilde \a(\ba) > 0\}$ and $\tPhi^-=-\tPhi^+$. Then $\tPhi=\tPhi^+ \sqcup \tPhi^-$. Note that $\Phi^\pm \subseteq \tilde \Phi^\mp$. Let $\ell: \tW \to \NN$ be the length function given by $\ell(\tw)=|\tPhi^- \cap \tw(\tPhi^+)|$. Let $\SS^a=\{s_{\tilde \a}; \tilde \a \in \tPhi, \ell(s_{\tilde \a})=1\}$ be the set of simple affine reflections. Notice that $(W^a, \SS^a)$ is a Coxeter system, and let $\leq$ be the associated Bruhat order on $\tW = W^a \rtimes \Omega$.

For $\tta=(\a, k) \in \tPhi$, let $U_{\tta}: \GG_a \to LG$ be the corresponding affine root subgroup, where $LG$ denotes the loop group associated to $G$. More precisely, $U_{\tta}(z)= u_\a(z t^k)$ for $z \in \kk$, where $u_\a: \GG_a \to G$ is the root subgroup of $\a$. We set $$I = I_G = T(\co_\brF) \prod_{\tta \in \tPhi^+} U_{\tta}(\kk) = T(\co_\brF) \prod_{\a \in \Phi^+} u_\a(t \co_\brF) \prod_{\b \in \Phi^+} u_{-\b}(\co_\brF),$$ which is called an Iwahori subgroup of $G(\brF)$.

\subsection{} \label{admissible} Let $v \in V = Y \otimes \RR$. We say $v$ is dominant if $\<v, \a\> \ge 0$ for each $\a \in \Phi^+$, and denote by $\bar v$ the unique dominant $W_0$-conjugate of $v$. Let $Y^+$ and $V^+$ be the set of dominant vectors in $Y$ and $V$ respectively. For $v, v' \in V$ we write $v' \le v$ if $v - v' \in \RR_{\ge 0} (\Phi^+)^\vee$.

Let $\s$ be the Frobenius automorphism of $G(\brF)$. We also denote by $\s$ the induced automorphism on the root datum $\car$. Then $\s$ acts on $V$ as a linear transformation of finite order which preserves $\ba$. For $\tw \in \tW$ there exists a nonzero integer $m$ such that $(\tw\s)^m=t^\xi$ for some $\xi \in Y$. Define $\nu_\tw = \xi / m \in V$, which does not depend on the choice of $m$.

Let $b \in G(L)$. We denote by $[b] = [b]_G = \{g\i b \s(g); g \in G(L)\}$ the $\s$-conjugate class of $b$. By \cite{Ko}, the $\s$-conjugacy class $[b]$ is determined by two invariants: the Kottwitz point $\k_G(b) \in \pi_1(G)_\s$ and the Newton point $\nu_G(b) \in (V^+)^\s$. Here $\k_G: G(\brF) \to \pi_1(G)_\s = \pi_1(G) / (\s-1)\pi_1(G)$ is the natural projection. To define $\nu_G(b)$, we note that there exists $\tw \in \tW$ such that $\tw \in [b]$. Then $\nu_G(b) = \bar \nu_{\tw}$, which does not depend on the choice of $\tw$. For $b' \in G(\brF)$ we set $$\JJ_{b, b'} = \JJ_{b, b'}^G =\{g \in G(\brF); g\i b \s(g)=b'\},$$ and put $\JJ_b = \JJ_{b, b'}$ if $b = b'$.

For $\l \in Y^+$ and $b \in G(\brF)$ we define $$X(\l, b) = X^G(\l, b)_I = \{g \in G(\brF) / I; g\i b \s(g) \in I \Adm(\l) I\},$$ where $\Adm(\l)$ is the $\l$-admissible set defined by $$\Adm(\l)=\{x \in \tW; x \leq t^{w(\l)} \text{ for some $w \in W_0$}\}.$$ Note that $\JJ_b$ acts on $X(\l, b)$ by left multiplication. By \cite{He2}, $X(\l, b) \neq \emptyset$ if and only if $\k_G(t^\l)=\k_G(b)$ and $ \nu_G(b) \le \l^\diamond$, where $t^\l := \l(t)$ and $\l^\diamond$ is the $\s$-average of $\l$. We say the pair $(\l, b)$ is Hodge-Newton irreducible if $X(\l, b) \neq \emptyset$ and $\l^\diamond - \nu_G(b) \in \RR_{> 0} (\Phi^+)^\vee$.

\subsection{} \label{parabolic} Let $M \supseteq T$ be a (semi-standard) Levi subgroup of $G$. Then $B \cap M$ is a Borel subgroup of $M$. By replacing the triple $(T, B, G)$ with $(T, B \cap M, M)$, we can define, as in previous subsections, $\Phi_M^+$, $\tW_M$, $\SS_M^a$, $\Omega_M$, $\tPhi_M^+$, $I_M$, $\k_M$ and so on.

For $v \in V$ we set $\Phi_v=\{\a \in \Phi; \a(v)=0\}$ and let $M_v \subseteq G$ be the Levi subgroup generated by $T$ and the root subgroups $u_\a$ for $\a \in \Phi_v$. We set $\tW_v = \tW_{M_v}$, $\tilde \Phi_v = \tilde \Phi_{M_v}$, and so on. If $v$ is dominant, let $J_v = \{s \in \SS_0; s(v)=v\}$.

Let $J \subseteq \SS_0$. Then there exists some $v' \in V^+$ such that $J_{v'} = J$, and we put $\Phi_J = \Phi_{M_{v'}}$, $\tW_J=\tW_{M_{v'}}$, $W_J^a = W_{M_{v'}^a}$, $\Omega_J=\Omega_{M_{v'}}$, and so on. We say $\mu \in Y$ is $J$-dominant (resp. $J$-minuscule) if $\<\a, \mu\> \ge 0$ (resp. $\<\a, \mu\> \in \{0, \pm1\}$) for $\a \in \Phi_J^+$.

Let $K \subseteq \SS^a$. Let $W_K \in W^a$ be the parabolic subgroup generated by $K$. Set ${}^K \tW = \{\tw \in \tW; \tw < s \tw \text{ for } s \in K\}$ and $\tW^K = ({}^K \tW)\i$. For $\tw \in {}^K \tW$ we define $I(K, \tw) = \max\{K' \subseteq K; \tw \s(K') \tw\i = K'\}$.

Let $\tw, \tw' \in \tW$ and $s \in \SS^a$. Write $\tw  \to_s \tw'$ if $\tw' = s \tw \s(s)$ and $\ell(\tw') \le \ell(\tw)$. For $K \subseteq \SS^a$ we write $\tw \to_K \tw'$ if there is a sequence $\tw = \tw_0 \to_{s_0} \tw_1 \to_{s_1} \cdots \to_{s_n} \tw_{n+1} = \tw'$ with $s_i \in K$ for $0 \le i \le n$.

\begin{lem} \label{K-min}
Let $K \subseteq \SS^a$ and $\tw \in {}^K \tW$. Then we have

(1) if $\tw < \tw s$ with $s \in \SS^a$, then either $\tw s \in {}^K \tW$ or $\tw s = s' \tw$ for some $s' \in K$;

(2) $\tw$ is the unique element of its $W_K$-$\s$-conjugacy class which lies in ${}^K \tW$.
\end{lem}

\begin{thm} \cite[\S 3]{He1} \label{partial}
Let $K \subseteq \SS^a$ and $\tw \in \tW$. Then there exist $x \in {}^K \tW$ and $u \in I(x, K)$ such that $\tw \to_K u x$. Moreover, such $x$ is uniquely determined by the $W_K$-$\s$-conjugacy class of $\tw$.
\end{thm}

\subsection{} \label{ss} We say $\tw \in \tW$ is semi-standard if $\tw\s(\tPhi_{\nu_{\tw}}^+) = \tPhi_{\nu_{\tw}}^+$, that is, ${}^{\tw\s} I_{M_{\nu_\tw}} = I_{M_{\nu_\tw}}$. We say $\tw$ is standard if it is semi-standard and $\nu_{\tw}$ is dominant. Let $\cs$ and $\cs^+$ denote the set of semi-standard elements and standard elements respectively.

By abuse of notation, we will freely identify an element of $\tW$ with its lift in $N_T(\co_\brF)$, according to the context.
\begin{lem} \label{semi}
Let $\tw \in \cs$. Then we have

(1) $z \tw \s(z)\i \in \cs$ if $z \in \tW$ such that $z(\tPhi_{\nu_{\tw}}^+) \subseteq \tPhi^+$;

(2) there exists a unique pair $(\tw', z') \in \cs^+ \times W_0^{J_{\bar \nu_{\tw}}}$ such that $\tw = z' \tw' \s(z')\i$;

(3) $s \tw \s(s)\i \in \cs$ if $s \in \SS^a$ and either $s \tw < \tw$ or $\tw \s(s) < \tw$;

(4) $\JJ_{\tw}$ is generated by $I \cap \JJ_{\tw}$ and $\tW \cap \JJ_{\tw}$.
\end{lem}
\begin{proof}
By assumption we have $z (\tPhi_{\nu_{\tw}}^+) = \tPhi_{\nu_{z \tw \s(z)\i}}^+$. So (1) follows by definition.

Let $z' \in W_0^{J_{\bar \nu_{\tw}}}$ such that $z'(\bar \nu_{\tw}) = \nu_{\tw}$. Let $\tw' = {z'}\i \tw \s(z')$. Note that $z'(\tPhi_{\bar \nu_{\tw}}^+) = \tPhi_{\nu_{\tw}}^+$. So $\tw' \in \cs^+$ by (1). Suppose there exists another pair $(\tw'', z'') \in \cs^+ \times  W_0^{J_{\bar \nu_{\tw}}}$ satisfies (2). Then $\nu_{\tw'} = \nu_{\tw''} = \bar \nu_{\tw}$ and ${z'}\i z'' \in W_{J_{\bar \nu_{\tw}}}$. Thus $z' = z'' \in W_0^{J_{\bar \nu_{\tw}}}$ and hence $\tw' = \tw''$.

By (1), to prove (3) it suffices to show $s(\tPhi_{\nu_{\tw}}^+) \subseteq \tPhi^+$. Otherwise, the simple affine root of $s$ lies in $\tPhi_{\nu_{\tw}}^+$. Hence $s \tw, \tw \s(s) > \tw$  (since $\tw\s(\tPhi_{\nu_{\tw}}^+) = \tPhi_{\nu_{\tw}}^+$), contradicting our assumption.

Note that $\JJ_\tw \subseteq M_{\nu_{\tw}}$. Thus (4) follows from that ${}^{\tw\s} I_{M_{\nu_\tw}} = I_{M_{\nu_\tw}}$, ${}^{\tw\s} \tW_{M_{\nu_\tw}} = \tW_{M_{\nu_\tw}}$, and the Bruhat decomposition $M_{\nu_\tw}(\brF) = I_{M_{\nu_\tw}} \tW_{M_{\nu_\tw}} I_{M_{\nu_\tw}}$.
\end{proof}

\section{Proof of Theorem \ref{intro}} \label{sec-main}
We fix $\l \in Y^+$ and $b \in G(\brF)$ such that $X(\l, b) \neq \emptyset$. Let $J = J_{\nu_G(b)} \subseteq \SS_0$. Moreover, we may and do assume that $b \in M_J(\brF)$ and $\nu_{M_J}(b) = \nu_G(b)$. In particular, $\JJ_b = \JJ_b^{M_J}$.

For $x \in \pi_1(M_J) \cong Y / \ZZ \Phi_J^\vee$ we set $\tw_x = t^{\mu_x} w_x \in \Omega_{M_J}$ with $\mu_x \in Y$ and $w_x \in W_J$ such that $\eta_{M_J}(\tw_x) = x$. Here $\eta_{M_J} : M_J(\brF) \to \pi_1(M_J)$ denotes the natural projection.
Define \begin{align*} \cs_{\l, b}^+ &= \{x \in \pi_1(M_J); \k_{M_J}(x) = \k_{M_J}(b), \mu_x \preceq \l\}, \\ \cs_{\l, b, x} &= \{z \tw_x \s(z)\i \in \Adm(\l); x \in \cs_{\l, b}^+, z \in W_0^J\}. \end{align*} Let $\cs_{\l, b}$ be the set of semi-standard elements in $\Adm(\l)$ which are $\s$-conjugate to $b$. Then we have $\cs_{\l, b} = \sqcup_{x \in \cs_{\l, b}^+} \cs_{\l, b, x}$ by Lemma \ref{semi}.
\begin{thm}[\cite{HZ}] \label{str}
Each connected component of $X(\l, b)$ intersects $\JJ_{b, \tw} I/I$ for some $\tw \in \cs_{\l, b}$.
\end{thm}
\begin{proof}
By \cite[Theorem 4.1]{HZ}, each connected component of $X(\l, b)$ intersects $\JJ_{b, \tw}$ for some $\s$-straight element $\tw \in \Adm(\l)$ which is $\s$-conjugate to $b$. Then the statement follows from the proof of \cite[Theorm 1.3]{N1}, which shows that $\s$-straight elements are semi-standard.
\end{proof}

For $g, g' \in G(\brF)$ we write $g I \sim_{\l, b} g' I$ if they are in the same connected component of $X(\l,b)$. For $\tw, \tw' \in \cs_{\l, b}$, we write $\JJ_{b, \tw} \sim_{\l, b} \JJ_{b, \tw'}$ if their natural images in $\pi_0(X(\l, b))$ coincide.
\begin{prop} \label{one-leaf}
For $x \in \cs_{\l, b}^+$ and $\tw, \tw' \in \cs_{\l, b, x}$ we have $\JJ_{b, \tw} \sim_{\l, b} \JJ_{b, \tw'}$.
\end{prop}

In the following four propositions, we retain the assumptions in Theorem \ref{intro}. The proofs are given in the remaining sections.
\begin{prop} \label{surj}
For $x, x' \in \cs_{\l, b}^+$ we have $\JJ_{b, \tw_x} \sim_{\l, b} \JJ_{b, \tw_{x'}}$. As a consequence, the natural projection $\JJ_{b, \tw} \to \pi_0(X(\l, b))$ is surjective for $\tw \in \cs_{\l, b}$.
\end{prop}

\begin{prop} \label{normal}
The natural action of $\ker(\eta_{M_J}) \cap \JJ_b$ on $\pi_0(X(\l, b))$ is trivial.
\end{prop}

\begin{prop} \label{ker-G}
The natural action of $$(\ker(\eta_G) \cap \JJ_b) / (\ker(\eta_{M_J}) \cap \JJ_b) \cong (\ZZ\Phi^\vee / \ZZ\Phi_J^\vee)^\s$$ on $\pi_0(X(\l, b))$ is trivial.
\end{prop}

\begin{proof}[Proof of Theorem \ref{intro}]
By \cite[Theorem 1.1]{He2}, it suffices to consider the Iwahori case $K = I$. Combining Proposition \ref{surj}, \ref{normal} and \& \ref{ker-G} we have $\pi_0(X(\l, b)) \cong \JJ_b / (\JJ_b \cap \ker(\eta_G))$. As $b$ is basic in $M_J$ and $\JJ_b = \JJ_b^{M_J}$, it follows that $\JJ_b$ is generated by $I_{M_J} \cap \JJ_b$ and $\tW_{M_J} \cap \JJ_b$. So we have $\JJ_b / (\ker(\eta_G) \cap \JJ_b) = (\Omega_J \cap \JJ_b) / (\Omega_J \cap \ker{\eta_G} \cap \JJ_b) = \Omega_J^\s / (\Omega_J^\s \cap \ker(\eta_G)) \cong \pi_1(G)^\s$ as desired.
\end{proof}

\section{The set $\cp_\tw$} \label{sec-P}

In the rest of the paper, we assume that $G$ is adjoint, simple, and its root system $\Phi$ has $d$ irreducible factors.

\subsection{} \label{sec-str} Let $\tw \in \Adm(\l)$. For $\a \in \Phi^+ \setminus \Phi_{\nu_\tw}$ we set $\a^i = (\tw\s)^i(\a) \in \tPhi$ for $i \in \ZZ$, and $$m_{\a, \tw} = \min\{ i \in \ZZ_{\ge 1}; \a^{-i} \in \tPhi \setminus \Phi\},$$ which is well defined since $\<\a, \nu_\tw\> \neq 0$.

We say $\a \in \Phi^+ \setminus \Phi_{\nu_\tw}$ is $\tw$-permissible if $\tw \s(s_\a) \in \Adm(\l)$ and $\a^{-m_{\a, \tw}} \in \tPhi^+$. Let $\cp_{\tw}$ denote the set of $\tw$-permissible roots.

Let $R \subseteq \SS_0$ be a $\s$-orbit of $\SS_0$. We say $\tw \in \Adm(\l)$ is left (resp. right) $R$-distinct if $s \tw \notin \Adm(\l)$ (resp. $\tw s \notin \Adm(\l)$) for $s \in R$. Let $w_R$ denote the longest root of $W_R$.
\begin{lem} \label{permissible}
Let $R$ be a $\s$-orbit of $\SS_0$, and let $\tw \in \Adm(\l)$ be left $R$-distinct. Then $w_R \tw w_R \in \Adm(\l)$ is right $R$-distinct. Moreover, $\cp_{w_R \tw w_R} \neq \emptyset$ if $\cp_\tw \neq \emptyset$.
\end{lem}
\begin{proof}
To show the first statement, we can assume $d = 1$, and hence $R$ is either commutative or is of type $A_2$. Then the first statement follows from Lemma \ref{R-dist}.

Now we show the ``Moreover" part. Let $\a \in \cp_\tw$ and let $$n_\a = \min \{i \in \ZZ_{\ge 0}; \a^{-i} \notin \Phi_R^+\} \le m_{\a, \tw}.$$ We show $w_R(\a^{-n_\a}) \in \cp_{w_R \tw w_R}$, and it suffices to check that $\a^{-n_\a} \in \Phi^+$ and \begin{align*} \tag{a} w_R \tw w_R \s(s_{w_R(\a^{-n_\a})}) = w_R \tw \s(s_{\a^{-n_\a}}) w_R \in \Adm(\l). \end{align*} If $n_\a = 0$, then $\a^{-n_\a} = \a \in \Phi^+ \setminus \Phi_R$, and (a) follows from Corollary \ref{conj}. Otherwise, $\a^{-n_\a+1} \in \Phi_R^+$, which implies that $\a^{-n_\a} = (\tw\s)\i (\a^{-n_\a+1}) \in \tPhi^-$ as $\tw$ is left $R$-distinct. Hence $\a^{-n_\a} \in \Phi^+$ since $n_\a \le m_{\a, \tw}$ and $\a^{-m_{\a, \tw}} \in \tPhi^+$. Moreover, $$w_R \tw w_R \s(w_R(\a^{-n_\a})) = w_R(\a^{-n_\a+1}) \in \Phi_R^-,$$ which means $w_R \tw w_R \s(s_{w_R(\a^{-n_\a})}) \leq w_R \tw w_R \in \Adm(\l)$, and (a) follows.
\end{proof}

\subsection{} \label{flat-subsec}
Fix $\eta \in Y$. Let $A = \max\{|\<\a, \eta\>|; \a \in \Phi\}$. Choose $M  \in \ZZ_{\ge 2}$ such that $M |\<\a, \eta\>| > 2A$ for any $\a \in \Phi$ with $\<\a, \eta\> \neq 0$. Let $\tw \in t^\mu W_0 \subseteq W_0 t^\eta W_0$, where $\mu$ is some $W_0$-conjugate of $\eta$. Motivated from the $\underline{\mathbf{a}}$-function in \cite{He0}, we define $$\nu_{\tw}^\flat = \sum_{i=0}^{N-1} \frac{p(\tw\s)^i(\mu)}{M^i},$$ where $N$ is the order of $p(\tw\s)$, and $p: \tW \rtimes \<\s\> \to W_0 \rtimes \<\s\>$ is the natural projection.
\begin{lem} \label{flat}
Let $\a \in \Phi$ and $0 \le n \le N-1$ such that $\<\a, p(\tw\s)^n(\mu)\> \neq 0$ and $\<\a, p(\tw\s)^i(\mu)\> = 0$ for $0 \le i \le n-1$. Then $\<\a, \nu_{\tw}^\flat\> \<\a, p(\tw\s)^n(\mu)\> > 0$.
\end{lem}
\begin{proof}
Note that $\mu, \eta$ are conjugate by $W_0$. By the choice of $M \ge 2$ we have $$|\frac{\<\a, p(\tw\s)^n(\mu)\>}{M^n}| > \frac{2A}{M^{n+1}} > \frac{A}{M^{n+1}} \sum_{i=n+1}^{N-1} \frac{1}{M^{i-n-1}} \ge \sum_{i=n+1}^{N-1} |\frac{\<\a, p(\tw\s)^i(\mu)\>}{M^i}|.$$ So the statement follows.
\end{proof}

\begin{cor} \label{dominant}
We have the following:

(1) $\<\a, \nu_{\tw}^\flat\> = 0$ if and only if $\<\a, p(\tw\s)^i(\mu)\> = 0$ for $i \in \ZZ$;

(2) $\nu_{\tw}^\flat$ is dominant for $\Phi_{\nu_{\tw}}^+$ if $\tw \in \cs$;

(3) $\nu_{z \tw \s(z)\i}^\flat = z(\nu_{\tw}^\flat)$ for $z \in W_0$;

(4) $\tw\s(\tPhi_{\nu_{\tw}^\flat}^\pm) = \tPhi_{\nu_{\tw}^\flat}^\pm$ if $\tw \in \cs$;

(5) if $\a \in \cp_\tw$, then $\<\a^i, \nu_\tw^\flat\> < 0$ for $1-m_{\a, \tw} \le i \le 0$, and the roots $\a^i \in \Phi$ for $1-m_{\a, \tw} \le i \le 0$ for are linearly independent. Here $\a^i = (\tw\s)^i(\a)$.
\end{cor}
\begin{proof}
The statement (1) follows from Lemma \ref{flat} and the definition of $\nu_{\tw}^\flat$.

Suppose there exists $\a \in \Phi_{\nu_{\tw}}^+$ such that $\<\a, \nu_{\tw}^\flat\> < 0$. By Lemma \ref{flat}, there exists $n \in \ZZ_{\ge 0}$ such that $\<\a, p(\tw\s)^n(\mu)\> < 0$ and $\<\a, p(\tw\s)^i(\mu)\> = 0$ for $0 \le i \le n-1$. In particular, we have $(\tw\s)^{-i}(\a) = p(\tw\s)^{-i}(\a)$ for $1 \le i \le n$ and $(\tw\s)^{-n-1}(\a) \in \tPhi^+ \setminus \Phi$, contradicting that $\tw \in \cs$. So (2) follows.

The statement (3) follows by definition.

By (1) we have $\tPhi_{\nu_{\tw}^\flat} = \tw\s(\tPhi_{\nu_{\tw}^\flat}) \subseteq \tPhi_{\nu_{\tw}}$. As $\tw \in \cs$, we have $\tw\s(\tPhi_{\nu_{\tw}}^\pm) = \tPhi_{\nu_{\tw}}^\pm$. So (4) follows from that $\tPhi_{\nu_\tw^\flat}^\pm = \tPhi_{\nu_\tw^\flat} \cap \tPhi_{\nu_\tw}^\pm$.

Let $\a \in \cp_\tw$ and let $m = m_{\a, \tw}$. By definition, $\<\a^{1-m}, \mu\> < 0$, $\a^{-i} = p(\tw\s)^{-i}(\a)$ and $\<\a^{1-i}, \mu\> = \<\a, p(\tw\s)^{i-1}(\mu)\> = 0 $ for $ 1 \le i \le m-1$. Thus it follows from Lemma \ref{flat} that $\<\a^i, \nu_\tw^\flat\> < 0$ for $1-m \le i \le 0$. Suppose $\sum_{i=0}^{1-m} c_i \a^i = 0$, where the coefficients $c_i \in \RR$ are not all zero. Let $i_0 = \min\{1-m \le i \le 0; c_i \neq 0\}$. Then $$0 = \<p(\tw\s)^{1-m - i_0}(\sum_{i=0}^{1-m} c_i \a^i), \mu \> = \sum_{i=0}^{i_0} c_i \<\a^{1-m-i_0+i}, \mu\> = c_{i_0} \<\a^{1-m}, \mu\> \neq 0,$$ which is a contradiction. So (5) follows.
\end{proof}

\begin{lem} \label{min}
Let $\tw \in \cs$ and let $z_0 \in W_0$ be the unique minimal element such that $z_0(\nu_{\tw}^\flat)$ is dominant. Then $z_0 \tw \s(z_0)\i \in {}^{\SS_0} \tW$. In particular, $\tw \in {}^{\SS_0} \tW$ if $\nu_{\tw}^\flat$ is dominant.
\end{lem}
\begin{proof}
Let $\tw' = z_0 \tw \s(z_0)\i \in  t^{\mu'} W_0$ with $\mu' = z_0(\mu)$. By the choice of $z_0$, we have $\nu_{\tw'}^\flat = z_0(\nu_{\tw}^\flat)$ is dominant and $z_0(\tPhi_{\nu_{\tw}^\flat}^\pm) = \tPhi_{\nu_{\tw'}^\flat}^\pm$. By Lemma \ref{flat}, $\mu'$ is dominant since $\nu_{\tw'}^\flat = z_0(\nu_{\tw}^\flat)$ is dominant. Let $\g \in \Phi^+$. We need to show that $\tw' < s_\g \tw'$. If $\<\g, \nu_{\tw'}^\flat\> > 0$, then either $\<\g, \mu'\> > 0$, or $\<\g, \mu'\> = 0$ and $\<p(\tw'\s)\i(\g), \nu_{\tw'}^\flat\> > 0$ (hence $p(\tw'\s)\i(\g) \in \Phi^+$) by Corollary \ref{dominant}, which means $\tw' < s_\g \tw'$ as desired. Suppose $\<\g, \nu_{\tw'}^\flat\> = 0$, that is, $\g \in \Phi_{\nu_{\tw'}^\flat}^+$. Then by Corollary \ref{dominant} (4) we have $$(\tw'\s)\i(\g) \in  z_0 (\tw\s)\i z_0\i(\tPhi_{\nu_{\tw'}^\flat}^-) = z_0 (\tw\s)\i (\tPhi_{\nu_{\tw}^\flat}^-) = z_0 (\tPhi_{\nu_{\tw}^\flat}^-) = \tPhi_{\nu_{\tw'}^\flat}^-.$$ So we also have $\tw' < s_\g \tw'$ as desired.
\end{proof}

\begin{prop} \label{unique}
Let $K \subseteq \SS_0$ and $\tw \in \cs_{\l, b}$. Then there exists a unique semi-standard element $\tw' \in {}^K \tW$ which is $\s$-conjugate to $\tw$ by $W_K$. If, moreover, $K = \SS_0$ and $(\l, b)$ is Hodge-Newton irreducible, then $\tw'$ is not left $R$-distinct for any $\s$-orbit $R$ of $\SS_0$.
\end{prop}
\begin{proof}
By Theorem \ref{partial}, there exist unique $\tw' \in {}^K \tW$ and some $u \in I(K, \tw')$ such that $\tw \to_K u \tw'$. So $\Phi_{I(K, \tw')} \subseteq \Phi_{\nu_{\tw'}}$ and $\ell(u \tw') = \ell(u) + \ell(\tw')$. As $\tw \in \cs$, we have $u \tw' \in \cs$ by Lemma \ref{semi} (2). So $u\tw' \leq u\i u\tw' = \tw'$, which means $u = 1$, and the first statement follows. The second one is proved in \cite[Lemma 6.11]{CN}.
\end{proof}

\subsection{} Let $\tw, \tw' \in \tW$ and $s \in \SS^a$. Write $\tw \rightharpoonup_s \tw'$ if $\tw' = s \tw \s(s)$ and $s \tw < \tw$ (which implies that $\ell(\tw') \le \ell(\tw)$). For $K \subseteq \SS^a$ we write $\tw \rightharpoonup_K \tw'$ if there is a sequence $\tw = \tw_0 \rightharpoonup_{s_0} \tw_1  \rightharpoonup_{s_1} \cdots  \rightharpoonup_{s_n} \tw_{n+1} = \tw'$ with $s_i \in K$ for $0 \le i \le n$.
\begin{lem} \label{finite-seq}
Let $K \subseteq \SS_0$ and $\tw \in \cs$. There is no infinite sequence $$\tw = \tw_0 \rightharpoonup_{s_0} \tw_1  \rightharpoonup_{s_1} \cdots,$$ where $s_i \in K$ for $i \in \ZZ_{\ge 0}$. As a consequence, we have $\tw \rightharpoonup_K \tw'$ for some $\tw' \in {}^K \tW$.
\end{lem}
\begin{proof}
We argue by induction on $|K|$. If $K = \emptyset$, the statement is trivial. Assume $|K| \ge 1$. Suppose there exists such an infinite sequence. By Lemma \ref{semi} we have $\tw_i \in \cs$ for $i \in \ZZ_{\ge 0}$. Moreover, by Lemma \ref{K-min} (1) we can assume that (a) $\ell(\tw_0) = \ell(\tw_1) = \cdots$, and (b) there exists $y \in {}^K \tW$ such that $\tw_i \in W_K y$ and $y \s(s_i) y\i \in K$ for $i \in \ZZ_{\ge 0}$. If each element of $K$ appears infinitely many times in the sequence $s_0, s_1, \dots$, then by (b) we have $K = I(K, y) \subseteq W_{\nu_\tw}$, see \S \ref{parabolic}. So $\tw_i = y \in {}^K \tW$ since $\tw_i \in \cs$, which is impossible. Otherwise, there exists a proper subset $K' \subsetneq K$ such that $s_i \in K'$ for $i \gg 0$, which contradicts the induction hypothesis. The proof is finished.
\end{proof}

Let $R$ be a $\s$-orbit of $\SS_0$. For $\tw, \tw' \in \tW$ we write $\tw \Rightarrow_R \tw'$ if $\tw, \tw' \in \cs$ are $W_R$-$\s$-conjugate and $\tw' \in {}^R \tW$. Write $\tw \Rightarrow \tw'$ if there is a sequence $\tw = \tw_0 \Rightarrow_{R_0} \tw_1 \Rightarrow_{R_1} \cdots \Rightarrow_{R_n} \tw_{n+1} = \tw'$.
\begin{prop} [{\cite[Proposition 6.16]{CN}}] \label{Left}
Let $\tw \in \cs$. Then $\tw \Rightarrow \tw'$, where $\tw' \in {}^{\SS_0} \tW$ is the unique element in the $W_0$-$\s$-conjugacy class of $\tw$.
\end{prop}
\begin{proof}
Assume otherwise. Then by Lemma \ref{finite-seq} there is an infinite sequence $$\tw = \tw_0 \rightharpoonup_{R_0} \tw_1 \rightharpoonup_{R_1} \cdots,$$ where $\tw_{i+1} \in {}^{R_i} \tW$ and $R_i$ is some $\s$-orbit of $\SS_0$ for $i \in \ZZ_{\ge 0}$. This contradicts Lemma \ref{finite-seq}. So the statement follows.
\end{proof}

\begin{lem} \label{existence}
Let $R$ be a $\s$-orbit of $\SS_0$. Let $\tw \in \Adm(\l) \cap \cs$. If $\tw \notin {}^R \tW$ and $\tw$ is not right $R$-distinct. Then $\cp_{\tw} \neq \emptyset$.
\end{lem}
\begin{proof}
By assumption, there exists $s \in R$ and $0 \le k \le |R|-1$ such that $\s^{-k}(s) \tw < \tw$, $\tw \s(s) \in \Adm(\l)$, and $$k = \min\{0 \le i \le |R|-1; \s^{-i}(s') \tw < \tw, \tw \s(s') \in \Adm(\l) \text { for some } s' \in R \}.$$ Let $\a \in \Phi^+$ be the simple root of $s$. We claim that \begin{align*} \tag{a} \a^{-i} =  \s^{-i}(\a) \text{ for } 0 \le i \le k, \text{ and hence } m_{\a, \tw} \ge k+1. \end{align*} Let $0 \le i \le k-1$. By the choice of $k$ we have $\tw < \s^{-i}(s) \tw$ and $\tw \s^{-i}(s) \notin \Adm(\l)$, which means $\s^{-i}(\a) = \tw \s^{-i}(\a)$ by Lemma \ref{R1}. So (a) is proved.

By (a) we have $\a^{-k} \in \Phi^+$. So $\a^{-k-1} = (\tw\s)\i (\a^{-k}) \in \tPhi^+$ since $\s^{-k}(s) \tw < \tw$. As $\tw \in \cs$, it follows that $\a^{-k} \notin \Phi_{\nu_\tw}$ and hence $\a^i \notin \Phi_{\nu_\tw}$ for $i \in \ZZ$. If $\a \notin \cp_\tw$, we have $\a^{-m_{\a, \tw}} \in \tPhi^- \setminus \Phi$ by definition, which means $\a^{-k-1} \in \tPhi^+ \cap \Phi = \Phi^-$. Let $\b = -\a^{-k-1} \in \Phi^+ \setminus \Phi_{\nu_\tw}$. Then $\b^{-m_{\b, \tw}} =  -\a^{-m_{\a, \tw}} \in \tPhi^+ \setminus \Phi$, and $\tw\s(s_\b) < \tw \in \Adm(\l)$ since $\tw\s(\b) = -\a^{-k} \in \Phi^-$. So $\b \in \cp_\tw$ as desired.
\end{proof}

\begin{cor} \label{non-empty}
Assume $(\l, b)$ is Hodge-Newton irreducible. For $\tw \in \cs_{\l, b}$ we have either $\tw \in {}^{\SS_0} \tW$ or $\cp_\tw \neq \emptyset$.
\end{cor}
\begin{proof}
By Proposition \ref{Left}, there exists a sequence $$\tw = \tw_0 \Rightarrow_{R_0} \tw_1 \Rightarrow_{R_1} \cdots \Rightarrow_{R_{n-1}} \tw_n = \tw',$$ where $\tw_0, \dots, \tw_{n+1}$ are distinct semi-standard elements, $R_0, \dots, R_n$ are $\s$-orbits of $\SS_0$, and $\tw' \in {}^{\SS_0}\tW$. We argue by induction on $n$. If $n = 0$, then $\tw \in {}^{\SS_0} \tW$ as desired. Assume $n \ge 1$. If $\tw = \tw_0$ is not right $R_0$-distinct, then $\cp_{\tw} \neq \emptyset$ by Lemma \ref{existence}. Otherwise, by Lemma \ref{R-dist}, $w_{R_0} \tw w_{R_0} \in \Adm(\l)$ is left $R_0$-distinct. So $w_{R_0} \tw w_{R_0} = \tw_1 \in {}^{R_0} \tW$ by Lemma \ref{K-min} (2). Moreover, $\tw_1 \notin {}^{\SS_0} \tW$ by Proposition \ref{unique}. By induction hypothesis, $\cp_{\tw_1} \neq \emptyset $, which implies $\cp_\tw \neq \emptyset$ by Lemma \ref{permissible}.
\end{proof}

\section{Proof of Proposition \ref{one-leaf}} \label{sec-leaf}
Assume $(\l, b)$ is Hodge-Newton irreducible. Recall that $d$ is the number of connected components of $\SS_0$. For $g \in G(\brF)$, $\tilde \g \in \tPhi$, $\tw \in \tW$, and $m \in \ZZ_{\ge 0}$, we define $$\textsl{g}_{g, \tilde\g, \tw, m}:  \PP^1 \to G(\brF) / I, \ z \mapsto g {}^{(\tw\s)^{1-m}} U_{\tilde\g}(z) \cdots {}^{(\tw\s)\i} U_{\tilde\g}(z) U_{\tilde\g}(z) I.$$

\begin{hyp} \label{hypo}
Recall that $\FF_q$ is the residue field of $F$. Assume that $q^d > 2$ (resp. $q^d > 3$) if some/any connected component of $\SS_0$ is non-simply-laced except of type $G_2$ (resp. is of type $G_2$).
\end{hyp}
Note that if Hypothesis \ref{hypo} is not true, then $d = 1$ and $\SS_0$ is non-simply-laced, which implies that $G$ is residually split, and hence split (since $G$ is unramified).

\begin{lem} \label{limit}
Suppose Hypothesis \ref{hypo} is true. Let $\tw \in \tW$, $\g \in \Phi$, and $m \in \ZZ_{\ge 0}$ such that the roots $\g^i := (\tw\s)^i(\g) \in \Phi$ for $1-m \le i \le 0$ are linearly independent. Let $\textsl{g} = \textsl{g}_{1, \g, \tw, m}$. Then there exist integers $1-m \le i_r < \cdots < i_0 \le 0$ such that $$\textsl{g}(\infty) = s_{\g^{i_r}} \cdots s_{\g^{i_0}} I, \text{ and } s_{\g^{i_0}} \cdots s_{\g^{i_{k-1}}} (\g^{i_k}) \in \Phi^+ \text{ for } 0 \le k \le r.$$ Moreover, if there exists $v \in V$ such that $\<\g^i, v\> < 0$ for $1-m \le i \le 0$, then $v \le (s_{\g^{i_r}} \cdots s_{\g^{i_0}})\i(v)$, where the equality holds if and only if $r < 0$, that is, $\g^i \in \Phi^-$ for $1-m \le i \le 0$.
\end{lem}
\begin{proof}
First notice that \begin{align*} \tag{a} {}^{(\tw\s)^i} U_\g(z) = U_{\g^i}(c_i z^{q^i}) \text{ with } c_i \in \co_\brF^\times \text{ for } 1-m \le i \le 0. \end{align*} We argue by induction on $m$. If $m = 0$, the statement is trivial. Assume $m \ge 1$. If $\g \in \Phi^-$, then $\textsl{g}(\infty) = \textsl{g}_{1, \g\i, \tw, m}(\infty)$, and it follows by induction hypothesis. Otherwise, we have $$\textsl{g}(z) = {}^{(\tw\s)^{1-m}} U_\g(z) \cdots {}^{(\tw\s)\i} U_\g(z) U_{-\g}(z\i) s_\g I \text{ for } z \neq 0.$$ As the roots $\g^i$ for $1-m \le i \le 0$ are linearly independent, it follows by (a) and induction on $m$ that $${}^{{}^{(\tw\s)^{1-m}} U_\g(z) \cdots {}^{(\tw\s)\i} U_\g(z)} U_{-\g}(z\i)= \prod_{(\b, a_\bullet)} U_\b(c_{a_\bullet} z^{n_{a_\bullet}}),$$ where $a_\bullet = (a_i)_{0 \le i \le m-1} \in (\ZZ_{\ge 0})^m$ such that $a_0 \ge 1$ and $a_i = 0$ unless $i \in d\ZZ$, $\b = -a_0 \g + \sum_{i=1}^{m-1} a_i \g^{-i} \in \Phi$, $c_{a_\bullet} \in \co_\brF$ and $n_{a_\bullet} = -a_0 + \sum_{i=1}^{m-1} a_i q^{-i}$. Moreover, we have $a_{jd} / a_0 \le 1$ (resp. $a_{j d} / a_0 \le 2$, resp. $a_{jd} / a_0 \le 3$) for $j \ge 1$ if some/any connected component of $\SS_0$ is simply-laced (resp. is non-simply-laced except of type $G_2$, resp. is of type $G_2$). Thus by Hypothesis \ref{hypo} we have $a_{jd} / a_0 \le q^d-1$ for $j \ge 1$, which implies that $n_{a_\bullet} < 0$ and $$\lim_{z \to \infty} {}^{{}^{(\tw\s)^{1-m}} U_\g(z) \cdots {}^{(\tw\s)\i} U_\g(z)} U_{-\g}(z\i) = 1.$$ Then $\textsl{g}(\infty) = s_{\g} \textsl{g}_{1, s_\g(\g\i), s_\g \tw \s(s_\g), m-1}(\infty)$, and the first statement follows by induction hypothesis.

Set $\b_k = s_{\g^{i_0}} \cdots s_{\g^{i_{k-1}}} (\g^{i_k}) \in \Phi^+$ and $v_k = s_{\g^{i_0}} \cdots s_{\g^{i_k}}(v)$ for $0 \le k \le r$. As $\<\g^{i_k}, v\> < 0$ we have $$v_k = s_{\b_k}(v_{k-1}) = v_{k-1} - \<\b_k, v_{k-1}\> \b_k^\vee = v_{k-1} - \<\g^{i_k}, v\> \b_k^\vee > v_{k-1}.$$ So the ``Moreover" part follows.
\end{proof}

Let $x \in \cs_{\l, b}^+$. Let $J_{x, 0}$ be union of connected components of $J$ on which $\s^i(\mu_x)$ is central for $i \in \ZZ$. Let $J_{x, 1} = J \setminus J_{x, 0}$. Let $H_x \subseteq M_J(\brF)$ be the subgroup generated by $I_{M_J}$, $W_{J_{x, 0}}$, and $W_{J_{x, 1}}^a$, see \S\ref{parabolic}. Note that $J_{x, 1}$ commutes with $J_{x, 0}$, and $\tw_x \in \tW_{J_{x, 1}}$.

Note that $\tW = \sqcup_{z \in W_0^J} z \tW_J = \sqcup_{z \in W_0^J} \sqcup_{\o \in \Omega_J} z \o\i W_J^a$.
\begin{lem} \label{tosemi}
Let $x \in \cs_{\l, b}^+$. Let $\tw \in \cs_{\l, b, x}$ and $z \in W_0^J$ such that $\tw = z \tw_x \s(z)\i$. Let $y \in \tW$ (resp. $y \in W_0$) such that $y \tw \s(y)\i \in \Adm(\l)$. Let $z' \in W_0^J$ and $\o \in \Omega_J$ such that $y z \in z' \o\i W_J^a$. Then $\tw' := z' \o\i \tw_x \s(\o) \s(z')\i \in \cs_{\l, b}$. Moreover, there exists $h \in \ker(\eta_{M_J}) \cap \JJ_{\tw_x}$ (resp. $h \in H_x \cap \JJ_{\tw_x}$) such that $g y\i I \sim_{\l, b} g z h \o {z'}\i I$ for $g \in \JJ_{b, \tw}$.
\end{lem}
\begin{proof}
Write $y z = z' \o\i u$ for some $u \in W_J^a$. Let $x' \in \pi_1(M_J)$ such that $\tw_{x'} = \o\i \tw_x \s(\o) \s(z')\i$. By \cite[Lemma 1.3]{CN} we have $$\tw' = z' \tw_{x'} \s(z')\i \leq z' \o\i \d \tw_x \s(\o) \s(z')\i = y \tw \s(y)\i \in \Adm(\l)$$ where $\d = u \tw_x \s(u)\i \tw_x\i \in W_J^a$. So $\tw' \in \Adm(\l)$ is semi-standard by Lemma \ref{semi}. In particular, $\tw_{x'} = \o\i \tw_x \s(\o) \s(z')\i \in \Adm(\l)$ and hence $x' \in \cs_{\l, b}^+$.

To show the``Moreover" part we set $$Z = \{m \in M_J(\brF) / I_{M_J}; m\i \tw_x \s(m) \in \cup_{\d' \leq_J \d} I_{M_J} \d' \tw_x I_{M_J}\}.$$ Note that $u\i I_{M_J} \in Z$. As $\tw_x \in \Omega_J$, by \cite[Theorem 4.1]{HZ} (resp. \cite[Lemma 6.13]{CN1}), there exists $h \in \ker(\eta_{M_J}) \cap \JJ_{\tw_x}$ (resp. $h \in H_x \cap \JJ_{\tw_x}$ if $y \in W_0$) such that $u\i I_{M_J},  h I_{M_J}$ are connected in $Z$. For $g \in \JJ_{b, \tw_x}$ there is an embedding $$Z \hookrightarrow X(\l, b), \quad m I_{M_J} \mapsto g z m \o {z'}\i I,$$ from which we have $g y\i I = g z u\i \o {z'}\i I \sim_{\l, b} g z h \o {z'}\i$ as desired.
\end{proof}

\begin{lem} \label{weyl-elementary}
Assume $G$ is not split. Let $x \in \cs_{\l, b}^+$ and $\tw \in \cs_{\l, b, x}$. If $\tw \notin {}^{\SS_0} \tW$, then there exist $h \in H_x \cap \JJ_{\tw_x}$ and $\tw' \in \cs_{\l, b, x}$ such that $\nu_{\tw}^\flat < \nu_{\tw'}^\flat$ and $g I \sim_{\l, b} g z h {z'}\i I$ for $g \in \JJ_{b, \tw}$. Here $z, z' \in W_0^J$ such that $\tw = z \tw_x \s(z)\i$ and $\tw' = z' \tw_x \s(z')\i$.
\end{lem}
\begin{proof}
By Corollary \ref{non-empty}, there exists $\a \in \cp_{\tw}$. Set $\a^i = (\tw\s)^i(\a)$ for $i \in \ZZ$.  Let $\textsl{g} = \textsl{g}_{g, \a, \tw, m_{\a, \tw}}$ for $g \in \JJ_{b, \tw}$. Since $\a^{-m_{\a, \tw}} \in \tPhi^+ \setminus \Phi$ and $\a^i \in \Phi$ for $1-m_{\a, \tw} \le i \le 0$, we have $$\textsl{g}\i b \s(\textsl{g}) \subseteq \tw U_{\s(\a)} I \subseteq I \{\tw\s, \tw \s(s_\a)\} I \subseteq I \Adm(\l) I.$$ As $G$ is not split, then Hypothesis \ref{hypo} is true. Moreover, by Lemma \ref{dominant} (5), the conditions in Lemma \ref{limit} are satisfied (for $(\g, m, v) = (\a, m_{\tw, \a}, \nu_\tw^\flat)$). Thus, by Lemma \ref{limit} we have $g I = \textsl{g}(0) \sim_{\l, b} \textsl{g}(\infty) = g y\i I$ for some $y \in W_0$ such that $y(v_\tw^\flat) > v_\tw^\flat$. Then $\tw'' := y \tw \s(y)\i \in \Adm(\l)$ and $\nu_{\tw}^\flat < y(\nu_{\tw}^\flat) = \nu_{\tw''}^\flat$. Let $h \in H_x$, $\tw' \in \cs_{\l, b, x}$, and $z' \in W_0^J$  be as in Lemma \ref{tosemi} such that $g I \sim_{\l, b} g y\i I \sim_{\l, b} g z h {z'}\i I$. By construction, $\tw'$ and $\tw''$ are $\s$-conjugate by $W_{\nu_{\tw'}} = z' W_J {z'}\i$, and hence $\nu_{\tw'}^\flat$ and $\nu_{\tw''}^\flat$ are conjugate by $W_{\nu_{\tw'}}$. By Corollary \ref{dominant} (2), $\nu_{\tw'}^\flat$ is dominant for $\Phi_{\nu_{\tw'}}^+$, which means $\nu_{\tw}^\flat < \nu_{\tw''}^\flat \le \nu_{\tw'}^\flat$ as desired.
\end{proof}

\begin{cor} \label{max}
Let $x \in \cs_{\l, b}^+$ and $\tw, \tw' \in \cs_{\l, b, x}$ with $\tw'$ the unique element in ${}^{\SS_0} \tW$. Then there exists $h \in H_x \cap \JJ_{\tw_x}$ such that $g I \sim_{\l, b} g z h {z'}\i I$ for $g \in \JJ_{b, \tw}$, where $z, z' \in W_0^J$ such that $\tw = z \tw_x \s(z)\i$ and $\tw' = z' \tw_x \s(z')\i$.
\end{cor}
\begin{proof}
Note that the statement follows from Theorem \ref{intro}, which is proved in \cite{CN} when $G$ is split. So we assume that $G$ is not split. If $\tw \notin {}^{\SS_0} \tW$, by Lemma \ref{weyl-elementary}, there exist $h \in H_x \cap \JJ_{\tw_x}$ and $\tw' \in \cs_{\l, b, x}$ such that $\nu_\tw^\flat < \nu_{\tw'}^\flat$ and $g I \sim_{\l, b} g z h {z'}\i I$ for $g \in \JJ_{b, \tw}$, where $z' \in W_0^J$ such that $\tw' = z' \tw_x \s(z')\i$. Repeating this process, we may assume either $\tw' \in {}^{\SS_0} \tW$ or $\nu_{\tw'}^\flat$ is dominant. In either case, we have $\tw' \in {}^{\SS_0} \tW$ by Lemma \ref{min}. So the statement follows.
\end{proof}

Proposition \ref{one-leaf} is a consequence of the following result.
\begin{prop} \label{weyl-conj}
Let $x \in \cs_{\l, b}^+$ and $\tw \in \cs_{\l, b, x}$. Then there exists $h \in H_x \cap \JJ_{\tw_x}$ such that $g I \sim_{\l, b} g h z\i I$, or equivalently, $g h\i I \sim_{\l, b} g z\i I$ for $g \in \JJ_{b, \tw_x}$, where $z \in W_0^J$ such that $\tw = z \tw_x \s(z)\i$. In particular, $\JJ_{b, \tw} \sim_{\l, b} \JJ_{b, \tw_x}$.
\end{prop}
\begin{proof}
Let $z' \in W_0^J$ such that $z' \tw_x \s(z')\i \in {}^{\SS_0} \tW$ (see Lemma \ref{semi} and Proposition \ref{unique}). By Corollary \ref{max}, there exist $h_1, h_2 \in H_x \cap \JJ_{\tw_x}$ such that $g I \sim_{\l, b} g h_1 {z'}\i I$ and $g z\i I \sim_{\l, b} g h_2 {z'}\i I$ for $g \in \JJ_{b, \tw_x}$. Then we have $$g h z\i I = j g z\i I \sim_{\l, b} j g h_2 {z'}\i I = g h_1 {z'}\i I \sim_{\l, b} g I,$$ where $h = h_1 h_2\i \in H_x \cap \JJ_{\tw_x}$ and $j =g h_1 h_2\i g\i \in \JJ_b$.
\end{proof}

\begin{cor} \label{pre-lin}
Let $x \in \cs_{\l, b}^+$ and $y \in \tW$ such that $y \tw_x \s(y)\i \in \Adm(\l)$. Then there exists $h \in \ker(\eta_{M_J}) \cap \JJ_{\tw_x}$ such that $g y\i I \sim_{\l, b} g h \o I$ for $g \in \JJ_{b, \tw_x}$, where $\o \in \Omega_J$ such that $y \in W_0^J \o\i W_J^a$.
\end{cor}
\begin{proof}
It follows from Lemma \ref{tosemi} and Proposition \ref{weyl-conj}.
\end{proof}

\section{Proof of Proposition \ref{surj}} \label{sec-surj}

\subsection{} \label{pre-K} Let $K \subseteq \SS_0$. Let $\tw = t^\mu w \in \Omega_K$ with $\mu \in Y$ and $w \in W_K$. Let $\g \in \Phi^+ \setminus \Phi_K$ such that $\g^\vee$ is $K$-dominant and $K$-minuscule. Set $\tilde \g = \g+1 \in \tPhi^+$. Suppose $$\mu, \mu - \g^\vee, \mu + w\s^r(\g^\vee), \mu - \g^\vee + w\s^r(\g^\vee) \preceq \l \text{ for some } r \in \ZZ_{\ge 0}.$$
\begin{lem} \label{orth}
Let $K$, $\tw =t^\mu w$, $\g$, $\tilde \g$, and $r$ be as in \S \ref{pre-K}. Then we have

(1) $\mu - \g^\vee, \mu + w(\s^r(\g^\vee)), \mu - \g^\vee + w\s^r(\g^\vee)$ are $K$-minuscule;

(2) $\tw, s_{\tilde \g} \tw, \tw s_{\s^r(\tilde \g)}, s_{\tilde \g} \tw s_{\tilde \g} \in \Adm(\l)$;

(3) $s_{\tilde \g} \tw s_{\s^r(\tilde \g)} \in \Adm(\l)$ if $\g \neq \s^r(\g)$ and $-\<w \s^r(\g), \mu\>, \<\g, \mu\> \le 1$;
\end{lem}
\begin{proof}
Note that (1) and (2) are proved in \cite[Lemma 4.4.6]{CKV} and \cite[Lemma 1.5]{CN1} respectively. To show (3) we claim that \begin{align*} \tag{a} \text{ there is a $W_K$-conjugate $\eta$ of $\mu$ such that $\eta - \g^\vee + \s^r(\g^\vee)$ is $K$-minuscule.} \end{align*} Indeed, let $\eta$ be a $W_K$-conjugate of $\mu$ such that $\eta - \g^\vee + \s^r(\g^\vee)$ is minimal under the partial order $\preceq$. If $\eta - \g^\vee + \s^r(\g^\vee)$ is not $K$-minuscule, then there exists $\a \in \Phi_K$ such that $\<\a, \eta - \g^\vee + \s^r(\g^\vee)\> \ge 2$. As $\eta$ is $K$-minuscule, and $\g^\vee, \s^r(\g^\vee)$ are $K$-dominant and $K$-minuscule, we deduce that $\<\a, \eta\> = 1$. Let $\eta' = s_\a(\eta) = \eta - \a^\vee$. Then we have $$\eta' - \g^\vee + \s^r(\g^\vee) = \eta - \g^\vee + \s^r(\g^\vee) - \a^\vee \prec \eta - \g^\vee + \s^r(\g^\vee),$$ which contradicts the choice of $\eta$. So (a) is proved.

By (1) and (a), $\eta - \g^\vee + \s^r(\g^\vee)$, $\mu - \g^\vee + w \s^r(\g^\vee)$ are conjugate by $W_K$. In particular, $\eta - \g^\vee + \s^r(\g^\vee) \preceq \l$. Then (3) follows from that $$s_{\tilde \g} \tw s_{\s^r(\tilde \g)} \leq s_{\tilde \g} t^\eta s_{\s^r(\tilde \g)} = s_\g t^{\eta - \g^\vee + \s^r(\g^\vee)} s_{\s^r(\g)} \leq t^{\eta - \g^\vee + \s^r(\g^\vee)} \in \Adm(\l),$$ where the first $\leq$ follows from \cite[Lemma 1.3]{CN}, and the second $\leq$ follows from that \begin{align*} \<\g, \s^r(\g^\vee)\> &\le 0 \text{ since } \g \neq \s^r(\g);  \\ \<\g, \eta - \g^\vee + \s^r(\g^\vee)\> &\le \<\g, \mu\> - 2 \le -1; \\ \<\s^r(\g), \eta - \g^\vee + \s^r(\g^\vee)\> &\ge \<w\s^r(\g), \mu\> + 2 \ge 1. \end{align*} The proof is finished.
\end{proof}

For $K \subseteq \SS_0$ we say $\g^\vee$ with $\g \in \Phi^+ \setminus \Phi_K$ is strongly $K$-minuscule if $\g^\vee$ is $K$-minuscule, and moreover, $\g$ is a long root if (1) some/any connected component of $\SS_0$ is of type $G_2$, and (2) $K$ is the set of short simple roots.
\begin{lem} \label{line}
Let $K$, $\tw =t^\mu w$, $\g$, $\tilde \g$, and $r$ be as in \S \ref{pre-K}. Assume furthermore that $\g^\vee$ is strongly $K$-minuscule. Then $U_{-\tilde \g} \tw U_{-\s^r(\tilde \g)}  \subseteq I \Adm(\l) I$ unless \begin{align*} \tag{*} \<\g, \mu\> = -\<w \s^r(\g), \mu\> = 1 \text{ and } \<\g, w\s^r(\g^\vee)\> = -1, \end{align*} in which case we have $$\tw \neq \tw', \ U_{-\s^r(\tilde \g)} \tw' U_{-\tilde \g}  \subseteq I \Adm(\l) I, \text{ and } \mu \pm (\g + w \s^r(\g))^\vee \preceq \l.$$ Here $\tw' = \mu - \g^\vee + \s^r(\g^\vee) \in \pi_1(M_K) \cong \Omega_K$.
\end{lem}
\begin{proof}
First we claim that \begin{align*} \tag{a} \text{$\Psi := \Phi \cap (\ZZ \g + \ZZ w\s^r(\g))$ is of type $A_2$, or $A_1 \times A_1$, or $A_1$.} \end{align*}

Otherwise, then $\Psi$ is of type $B_2$ or $G_2$. In particular, $\g = \s^r(\g)$ (since $\s^d = \Id$), $\g \neq w\s^r(\g) = w(\g)$, and hence $K \neq \emptyset$. If $\Psi$ is of $B_2$, then $\g \pm w\s^r(\g) \in \Phi$ and $\<\g, w\s^r(\g^\vee)\> = 0$ since $\g, w\s^r(\g)$ are of the same length. Thus $\g - w\s^r(\g) \in \Phi_K$ and $\<\g - w\s^r(\g), \g^\vee\> = 2$, contradicting that $\g^\vee$ is $K$-minuscule. So $\Psi$ is of type $G_2$. As $\g^\vee$ is strongly $K$-minuscule, $\g \neq w\s^r(\g)$ are short roots and $K$ consists of long simple roots, which contradicts that $\g^\vee$ is $K$-minuscule. So (a) is proved.

Then we claim that
\begin{align*} \tag{b} & \text{$U_{-\tilde \g} \tw U_{-\s^r(\tilde \g)}  \subseteq I \Adm(\l) I$ if one of the following holds:} \\  \tag{b1} & \text{either $\<\g, \mu\> \ge 2$ or $\<\g, \mu\> = 1$ and $\<\g, w\s(\g^\vee)\> \ge 0$;} \\  \tag{b2} & \text{either $\<w\s^r(\g), \mu\> \le -2$ or $\<w\s^r(\g), \mu\> = -1$ and $\<\g, w\s^r(\g^\vee)\> \ge 0$.} \end{align*} By symmetry we assume (b1) occurs. Then $U_{-\tw\i(\tilde \g)}, [U_{-\tw\i(\tilde \g)}, U_{-\s^r(\tilde \g)}] \subseteq I$ by (a). Thus $$U_{-\tilde \g} \tw U_{-\s^r(\tilde \g)} \subseteq \tw U_{-\s^r(\tilde \g)} I \subseteq I \{\tw, \tw s_{\s^r(\tilde \g)}\} \subseteq I \Adm(\l) I,$$ where the last inclusion follows from Lemma \ref{orth} (2). So (b) is proved.

Suppose $U_{-\tilde \g} \tw U_{-\s^r(\tilde \g)} \nsubseteq I \Adm(\l) I$. Then $-\<w \s^r(\g), \mu\>, \<\g, \mu\> \le 1$ by (b), and $\tw\i(\g) \neq \s^r(\g)$. Assume $\<\g, \mu\> \le 0$. Then $U_{\tw\i(\tilde \g)}, [U_{\tw\i(\tilde \g)}, U_{-\s^r(\tilde \g)}] \subseteq I$ by (a) and that $\g, \s^r(\g)$ are $K$-dominant. Thus by Lemma \ref{orth} we have $$U_{-\tilde \g} \tw U_{-\s^r(\tilde \g)} \subseteq I s_{\tilde \g} \tw U_{-\s^r(\tilde \g)} I \subseteq I \{ s_{\tilde \g}\tw, s_{\tilde \g} \tw s_{\s^r(\tilde \g)} \} I  \subseteq I \Adm(\l) I,$$ which contradicts our assumption. So $\<\g, \mu\> = 1$, and $\<w\s^r(\g), \mu\> = -1$ by symmetry. Moreover, we have $\<\g, w\s^r(\g^\vee)\> = -1$ by (b) and (a).

Write $\tw' = t^{\mu'} w' \in \Omega_K$ with $\mu' \in Y$ and $w' \in W_K$. Then $\mu', \mu - \g^\vee + w\s^r(\g^\vee)$ (resp. $\mu' - \s^r(\g^\vee), \mu - \g^\vee$, resp. $\mu' + w'(\g^\vee), \mu + w\s^r(\g^\vee)$) are conjugate by $W_K$ by Lemma \ref{orth} (1). Since $\<\g, \mu\> = -\<w\s^r(\g), \mu\> = -\<\g, w\s^r(\g^\vee)\> = 1$, it follows that $\mu - \g^\vee + w\s^r(\g^\vee)$ and $\mu \pm (\g^\vee + w\s^r(\g^\vee))$ are conjugate by $W_0$. Hence $\mu \pm (\g^\vee + w\s^r(\g^\vee)), \mu' \preceq \l$. As $w_K(\g)$ (with $w_K$ the longest element of $W_K$) is $K$-anti-dominant, we have $$\<w'(\g), \mu'\> = \<w_K(\g), \mu'\> \le \<\g, \mu - \g^\vee + w\s^r(\g^\vee)\> = -2.$$ Hence $\s^r(\g) \neq \g$, that is, $\tw \neq \tw'$, and $U_{-\s^r(\tilde \g)} \tw' U_{-\tilde \g}  \subseteq I \Adm(\l) I$ by (b2).
\end{proof}

\subsection{} Let $x, x' \in \cs_{\l, b}^+ \subseteq \pi_1(M_J)$. Write $x \overset {(\g, r)} \to x'$ for some $\g \in \Phi \setminus \Phi_J$ and $r \in \ZZ_{\ge 1}$ if $x' - x = \s^r(\g^\vee) - \g^\vee$ and $\mu_{x - \g^\vee}, \mu_{x + \s^r(\g^\vee)} \preceq \l$, see \S \ref{sec-main}. Moreover, write $x \overset {(\g, r)} \rightarrowtail x'$ if $x \overset {(\g, r)} \to x'$, and for each $1 \le i \le r-1$ we have \begin{align*}  &\text{neither } \quad\  x \overset {(\g, i)} \to x - \g^\vee + \s^i(\g^\vee) \overset {(\s^i(\g), r-i)} \to x', \\ &\text{ nor } \quad\   x \overset {(\s^i(\g), r-i)} \to x - \s^i(\g^\vee) + \s^r(\g^\vee) \overset {(\g, i)} \to x'. \end{align*} Notice that $x \overset {(\g, r)} \to x'$ is equivalent to $x' \overset {(-\g, r)} \to x$.
\begin{lem} [{\cite[Remark 4.5.2]{CKV}}] \label{saturate}
Let $x \neq x' \in \cs_{\l, b}^+$ such that $x \overset {(\g, r)} \rightarrowtail x'$ for some $\g \in \Phi \setminus \Phi_J$ and $r \in \ZZ_{\ge 1}$. Then $\tw_x \s^i(\d) = \s^i(\d)$ for any $W_0$-conjugate $\d$ of $\g$ and $1 \le i \le r-1$ with $i, i-r \notin d\ZZ$.
\end{lem}

For $\g \in \Phi$ we denote by $\co_\g$ the $\s$-orbit of $\g$.
\begin{prop} [{\cite[Lemma 6.7]{N2}}] \label{conneted}
Let $x \neq x' \in \cs_{\l, b}^+$. Then there exist distinct elements $x = x_0, x_1, \dots, x_m = x' \in \cs_{\l, b}^+$ such that for each $1 \le i \le m$ we have

(1) $x_{i-1} \overset {(\g_i, r_1)} \rightarrowtail x_i$ with $\g_i \in \Phi \setminus \Phi_J$ such that $\g_i^\vee$ $J$-dominant and $J$-minuscule;

(2) $1 \le r_i \le d-1$ if $|\co_{\g_i}| =d$; $1 \le r_i \le d$ if $|\co_{\g_i}| = 2d$; $1 \le r_i \le 2d-1$ if $|\co_{\g_i}| \le 3d$.
\end{prop}

\begin{proof}[Proof of Proposition \ref{surj}]
The case that $\s$ has order $3d$ is handled in \S \ref{subsec-III}. We consider the case that $\s$ has order $\le 2d$. Without loss of generality, we can assume that $|\co_\g| = 2d$. By Proposition \ref{conneted} and symmetry, we may assume $x \overset{(\g, r)} \to x'$ for some $1 \le r \le d$ and $\g \in \Phi^+ \setminus \Phi_J$ with $\g_i^\vee$ $J$-dominant and $J$-minuscule. Moreover, we can assume \begin{align*} \tag{a} U_{-\tilde \g} \tw_x U_{-\s^r(\tilde \g)} \subseteq I \Adm(\l) I. \end{align*} Indeed, if $1 \le r \le d-1$, (a) follows from Lemma \ref{orth} (2). If $r = d$, by Lemma \ref{line} we can switch the pairs $(x, \g)$ and $(x', \s^d(\g))$ if necessary so that (a) still holds.

Now we can assume further that $x \overset {(\g, r)} \rightarrowtail x'$. Let $\tilde \g = \g + 1 \in \tPhi^+$, and let $\textsl{g} = \textsl{g}_{g, -\s^{r-1}(\tilde \g), \tw_x, r}$ for $g \in \JJ_{b, \tw_x}$ (see \S\ref{sec-leaf}). By Lemma \ref{saturate}, $(\tw\s)^i(\g) = \s^i(\g)$ for $1 \le i \le r-1$. Then by (a) we have $\textsl{g}\i \tw \s(\textsl{g}) \subseteq U_{-\tilde\g} \tw_x U_{-\s^r(\tilde\g)} \subseteq I \Adm(\l) I$, which means that $gI = \textsl{g}(0) \sim_{\l, b} \textsl{g}(\infty) = g s I$, where $s = s_{\tilde \g} \cdots s_{\s^{r-1}(\tilde \g)}$. By \cite[Lemma 1.3]{CN1} we can write $s = \o z\i$, where $z \in W_0^J$ and $\o = \g^\vee + \cdots + \s^{r-1}(\g^\vee) \in \Omega_J \cong \pi_1(M_J)$. By Proposition \ref{weyl-conj}, there is $h' \in \JJ_{b, \tw_{x'}}$ such that $g I \sim_{\l, b} g \o z\i I \sim_{\l, b} g \o h' I$. So we have $\JJ_{b, \tw_x} \sim_{\l, b} \JJ_{b, \tw_{x'}}$ as desired.
\end{proof}

\section{Proof of Proposition \ref{normal}} \label{sec-normal}
Retain the assumptions and notations in previous sections.

For $K \subseteq \SS_0$ we denote by $\pr_K: \RR \Phi^\vee \to (\RR \Phi_K^\vee)^\perp$ the orthogonal projection.
\begin{lem} \label{pr}
Let $x \in \cs_{\l, b}^+$ and let $\co$ be a $\s$-orbit of $J$-anti-dominant roots in $\Phi^+ \setminus \Phi_J$. Then we have (1) $\sum_{\a \in \co} \<\a, \pr_J(\mu_x)\> > 0$, and (2) $\<w_J(\b), \mu_x\> \ge 1$ for some $\b \in \co$. Here $w_J$ denotes the longest element of $W_J$.
\end{lem}
\begin{proof}
Let $\g \in \co$. By definition, $\<\g, \nu_G(b)\> = \<\g, \pr_J(\mu_x)^\diamond\> > 0$. So (1) follows as $$\sum_{\a \in \co} \<\a, \pr_J(\mu_x)\> =  \sum_{\a \in \co} \<\a, \pr_J(\mu_x)^\diamond\> = |\co| \<\g, \nu_G(b)\> > 0.$$ By (1), there exists $\b \in \co$ such that $\<\b, \pr_J(\mu_x)\> > 0$. As $w_J(\b)$ is $J$-dominant and $\mu_x - \pr_J(\mu_x) \in \RR_{\ge 0} (\Phi_J^+)^\vee$, we have $$\<w_J(\b), \mu_x\> \ge \<w_J(\b), \pr_J(\mu_x)\> = \<\b, \pr_J(\mu_x)\> > 0.$$ So (2) follows.
\end{proof}

\begin{lem} [{\cite[Lemma 1.6]{CN1}}] \label{anti}
Let $K \subseteq \SS_0$ and $\tw = t^\mu w \in \Omega_K$ with $\mu \in Y$ and $w \in W_K$. Let $\a \in \Phi^+$ be $K$-anti-dominant. Then (1) $\tw s_\a \in \Adm(\l)$ if $\mu + \a^\vee \preceq \l$; (2) $s_\a \tw \in \Adm(\l)$ if $\mu - w(\a)^\vee \preceq \l$; (3) $z \tw z\i \in \Adm(\l)$ for $z \in \tW^K$.
\end{lem}

\subsection{} Define $J_1 = \cup_{x \in \cs_{\l, b}^+} J_{x, 1}$ and $J_0 = J \setminus J_1$. Define $H_{J'} = M_{J'}(\brF) \cap \ker(\eta_{M_{J'}})$ for $J' \subseteq J$.

\begin{thm}[{\cite[Theorem 6.3]{HZ}}] \label{stab-basic}
Let $x \in \cs_{\l, b}^+$. Then $H_{J_{x, 1}} \cap \JJ_{\tw_x}$ fixes each connected component of $X^{M_{J_{x, 1}}}(\mu_x, \tw_x)$.
\end{thm}

%\begin{thm}[{\cite[Theorem 3.3.1]{ZZ}}] \label{stab-irr} The stabilizer of each irreducible component of $X_{\tw}(b)$ in $\JJ_b$ is a parahoric subgroup. \end{thm}

\begin{lem} \label{J1}
We have that $H_{J_1} \cap \JJ_b$ fixes each connected component of $X(\l, b)$.
\end{lem}
\begin{proof}
Let $C$ be a connected component of $X(\l, b)$. Let $x \in \cs_{\l, b}^+$. By Proposition \ref{surj}, there exists $g \in \JJ_{b, \tw_x} \subseteq M_J(\brF)$ such that $g I \in C$. Moreover, $g I$ also lies in the image of the embedding $$X^{M_{J_{x, 1}}}(\mu_x, \tw_x) \hookrightarrow X(\l, b), \ h I_{M_{J_{x, 1}}} \mapsto g h I.$$ Thus $g (H_{J_{x, 1}} \cap \JJ_{\tw_x}) g\i = H_{J_{x, 1}} \cap \JJ_b$ fixes $C$ by Theorem \ref{stab-basic}. So the statement follows by noticing that $H_{J_1} \cap \JJ_b$ is generated by $H_{J_{x, 1}} \cap \JJ_b$ for $x \in \cs_{\l, b}^+$.
\end{proof}

\subsection{}
Let $K \subseteq J_0$ be the union of some $\s$-orbit of connected components of $J_0$.
\begin{lem} \label{J01}
If $\mu_x + \a^\vee \preceq \l$ for some $x \in \cs_{\l, b}^+$ and $\a \in K$, then $H_K \cap \JJ_b$ fixes each connected component of $X(\l, b)$.
\end{lem}
\begin{proof}
As $\mu_x$ is central on $\Phi_K$, we can assume $\a$ is $K$-dominant and hence $\s^d(\a) = \a$. Let $C$ be a connected component of $X(\l, b)$. Then $g I \in C$ for some $g \in \JJ_{b, \tw_x}$. So the stabilizer of $C$ in $\JJ_b$ equals $g Q g\i$, where $Q \subseteq \JJ_{\tw_x}$ a standard parahoric subgroup containing $I_{M_J} \cap \JJ_{\tw_x}$. By Lemma \ref{semi} (4) it remains to show $W_K^a \cap \JJ_{\tw_x} \subseteq Q$. Let $\textsl{g} = \textsl{g}_{g, \a, \tw_x, d}$ and $\textsl{g}' = \textsl{g}_{g, -\a-1, \tw_x, d}$ (see \S\ref{sec-leaf}). By Lemma \ref{anti} and Lemma \ref{orth}, $$g\i b \s(g) \subseteq U_\a \tw_x \subseteq I \Adm(\l) I, \text{ and } {g'}\i b \s(g') \subseteq \tw_x U_{-\a-1} \subseteq I \Adm(\l) I,$$ which means $$g s I = \textsl{g}(\infty) \sim_{\l, b} \textsl{g}(0) = \textsl{g}'(0) \sim_{\l, b} \textsl{g}'(\infty) = g s'I,$$ where $s = s_\a \cdots s_{\s^{d-1}(\a)}, s' = s_{\a+1} \cdots s_{\s^{d-1}(\a) + 1} \in \JJ_{\tw_x}$. So we have $s, s' \in Q$, which means $W_K^a \cap \JJ_{\tw_x} \subseteq Q$ since $\a \in \Phi_K^+$ is $K$-dominant.
\end{proof}

The following technical lemma is proved in \S \ref{subsec-weak}
\begin{lem} \label{weak}
If $\mu_{x''} + \d^\vee \npreceq \l$ for any $x'' \in \cs_{\l, b}^+$ and $\d \in K$, then there exist $x \in \cs_{\l, b}^+$ and $\b \in \Phi^+ \setminus \Phi_J$ with $\b^\vee$  $J$-anti-dominant and $J$-minuscule such that

(1) $\mu_x + \b^\vee \preceq \l$, and $\b^\vee$ is non-central on $K$;

(2) $\tw_x \s^i(\b) = \s^i(\b)$ for $i \in \ZZ \setminus n\ZZ$;

(3) $\<w_x \s^n(\b), \mu_x\> \ge 1$;

(4) if $\s^n$ does not act trivially on $\Psi_\b \cap J_0$, then $\Psi = \Phi$, $\Psi_\b$ is of type $E_6$, $\Psi_\b \cap J_0 = \{\a_1, \a_6\}$, $\Psi_\b \cap J_1 = \{\a_2, \a_4\}$, $\b = \a_3$, $\mu_x |_{\Psi_\b} = \o_4^\vee - \o_3^\vee$, and $\mu_x |_{\Psi \setminus \Psi_\b} = 0$.

Here, $n \in \{d, 2d, 3d\}$ denotes the minimal integer such that $\b, \s^n(\b)$ are in the same connected component $\Psi_\b$ of $\Psi := \Phi \cap \ZZ(J \cup \co_\b)$, whose simple roots $\a_i$ and fundamental coweights $\o_i^\vee$ for $1 \le i \le 6$ are labeled as in \cite{Hum}.
\end{lem}

\begin{lem} \label{contain}
Retain the situation of Lemma \ref{weak}. Let $\a \in \Phi_K^+$ such that $\<\a, \b^\vee\> = -1$. If $\a = \s^n(\a)$, then $U_\b \tw_x U_{\s^n(\b)}, U_\a s_\b \tw_x s_{\s^n(\b)} U_\a \subseteq I \Adm(\l) I$.
\end{lem}
\begin{proof}
Note that $\mu_x + s_\a(\b)^\vee = \mu_x + \b^\vee + \a^\vee \preceq \l$, $s_\a, s_\b, s_{s_\a(\b)} \in W^{J_1}$, and $\tw_x \in \Omega_{J_1}$. By Lemma \ref{weak} (2) and Lemma \ref{anti}, $s_\b \tw_x, s_{s_\a(\b)}\tw_x \in \Adm(\l)$. As $\tw_x\s^n(\b) \in \tPhi^+ \setminus \Phi$, we have $$U_\b \tw_x U_{\s^n(\b)} \subseteq I U_\b \tw_x, \ U_\a s_\b\tw_x s_{\s^n(\b)} U_{\a} \subseteq I U_\a s_\b\tw_x s_{\s^n(\b)},$$ and it remains to show $s_\a s_\b\tw_x s_{\s^n(\b)}, s_\b\tw_x s_{\s^n(\b)} \in \Adm(\l)$. As $\tw_x\s^n(\b) \in \tPhi^+ \setminus \Phi$, $\tw_x(\a) = \a$ and $s_\a s_\b(\a) \in \Phi^+$, we have $s_\b \tw_x s_{\s^n(\b)} \leq s_\b \tw_x \in \Adm(\l)$ and $$s_\a s_\b \tw_x s_{\s^n(\b)} \leq s_\a s_\b \tw_x \leq s_\a s_\b s_\a \tw_x = s_{s_\a(\b)} \tw_x \in \Adm(\l).$$ The proof is finished.
\end{proof}

\begin{proof}[Proof of Proposition \ref{normal}]
Let $K$ be the union of some $\s$-orbit of connected components of $J_0$. By Lemma \ref{J1} and Lemma \ref{J01}, it remains to show $H_K \cap \JJ_b$ acts trivially on $\pi_0(X(\l, b))$. Let $x$, $\b$ and $n$ be as in Lemma \ref{weak}. Let $g \in \JJ_{b, \tw_x}$ and $I_{M_J} \cap \JJ_{\tw_x} \subseteq Q \subseteq \JJ_{\tw_x}$ be as in the proof of Lemma \ref{J01}. It suffices to show $W_K^a \cap \JJ_{\tw_x} \subseteq Q$.

\

Case(1): $\s^n$ acts trivially on $\Psi_\b \cap J_0$. Let $\a \in \Phi_K^+$ be a highest root such that $\<\a, \b^\vee\> = -1$. Then it suffices to show $s, s' \in Q$, where $s = s_\a \cdots s_{\s^{n-1}(\a)}, s' =  s_{\a+1} \cdots s_{\s^{n-1}(\a) + 1}  \in \JJ_{\tw_x}$.

Let $r = s_\b \cdots \s^{n-1}(s_\b)$. We claim that \begin{align*}\tag{a} g I \sim_{\l, b} g r I \sim_{\l, b} g r s I \sim_{\l, b} g s I, \text{ and hence } s \in Q \end{align*}

To show the first relation $\sim_{\l, b}$ in (a) we define $\textsl{g} = \textsl{g}_{g, \s^{n-1}(\b), \tw_x, n}$. By Lemma \ref{weak} (2) and Lemma \ref{contain} we have $$\textsl{g}\i b \s(\textsl{g}) \subseteq U_\b \tw_x U_{\s^n(\b)} \subseteq I \Adm(\l) I,$$ which means $g I = \textsl{g}(0) \sim_{\l, b} \textsl{g}(\infty) = g r I$ as desired. The last relation $\sim_{\l, b}$ in (a) follows the same way by replacing $g$, $\b$ with $g s$, $s_\a(\b)$ respectively.

To show the second relation $\sim_{\l, b}$ in (b) we define $\textsl{g}' = \textsl{g}_{g r, \s^{n-1}(\a), \tw_x, n}$. Notice that $r\i \tw_x \s(r) = s_\b \tw_x s_{\s^n(\b)}$. Then by Lemma \ref{contain} we have $${\textsl{g}'}\i b \s(\textsl{g}') \subseteq U_\a  s_\b \tw_x s_{\s^n(\b)} U_\a  \subseteq I \Adm(\l) I,$$ which means $g r I = \textsl{g}'(0) \sim_{\l, b} \textsl{g}'(\infty) = g r s I$. So (a) is proved.

By Lemma \ref{J1}, Lemma \ref{J01}, and (a) we have $(W_{J_1}^a W_{J_0}) \cap \JJ_{\tw_x} \subseteq Q$, and hence \begin{align*} \tag{b} H_x \cap \JJ_{\tw_x} \subseteq Q. \end{align*}

Let $x' = x + \b^\vee - \s^n(\b)^\vee \in \pi_1(M_J)$. If $\b \neq \s^n(\b)$, then $\b, \s^n(\b)$ are neighbors of $\Psi_\b \cap K$ on which $\s^n$ acts trivially, which means they are in distinct connected components of $\Psi_\b \setminus K$. Thus \begin{align*} \tag{c} \<w(\b), \s^n(\b)^\vee\> = 0 \text { for any } w \in W_{J_1} \text{ if } \b \neq \s^n(\b^\vee). \end{align*} By Lemma \ref{weak} (1) \& (3) and (c) we have $x' \in \cs_{\l, b}^+$. Moreover, $\mu_x + \b^\vee - w_x(\b^\vee), \mu_{x'}$ are conjugate by $W_{J_1}$ as they are conjugate by $W_J$ and $\mu_{x'}$ is central on $J_0$. Let $\g_1 = w_{J_1}(\b)$ and $\g_2 = w_{J_1}(s_\a(\b))$ which are $J_1$-dominant. By Lemma \ref{orth} (1) and that $\s^n$ acts trivially on $\Psi_\b \cap J_0$, $$\mu_x, \mu_x - \s^n(\g_i^\vee), \mu_x + w_x(\g_i^\vee), \mu_x - \s^n(\g_i^\vee) + w_x(\g_i^\vee) \preceq \l$$ are conjugate to $$\mu_{x'} - \g_i^\vee + w_{x'} \s^n(\g_i^\vee), \mu_{x'} - \g_i^\vee, \mu_{x'} + w_{x'} \s^n(\g_i^\vee), \mu_{x'} \preceq \l$$ under $W_{J_1}$ respectively.

Let $\t = \b^\vee + \cdots \s^{n-1}(\b)^\vee \in \pi_1(M_{J_1}) \cong \Omega_{J_1}$. Then $\tw_x = \t\i \tw_{x'} \s(\t)$ and hence $g \t\i \in \JJ_{b, \tw_{x'}}$. Define $\textsl{g}_i = \textsl{g}_{g\t\i, -\s^{n-1}(\g_i)-1, \tw_{x'}, n}$. As $J_0 \neq \emptyset$, $\g_i^\vee$ is strongly $J_1$-minuscule. Then it follows from Lemma \ref{weak} (2), Lemma \ref{line} and (c) that $${\textsl{g}_i}\i b \s(\textsl{g}_i) \subseteq U_{-\g_i-1} \tw_{x'} U_{-\s^n(\g_i)-1} \subseteq I \Adm(\l) I,$$ which means $g \t\i I = \textsl{g}_i(0) \sim_{\l, b} \textsl{g}_i(\infty) = g \t\i s_i I$, where $s_i = s_{\g_i+1} \cdots s_{\s^{n-1}(\g_i) + 1}$. As $\g_i^\vee$ is $J_1$-minuscule and $J_1$-dominant, we have $s_i = \t_i y_i\i$, where $\t_i \in \Omega_{J_1}$ and $y_i \in W_0$. Notice that $g \t\i \t_i \in \JJ_{b, \tw_x}$, $\t = \t_1$, and $\t\i \t_2 = s' s \in \JJ_{\tw_x}$. By Lemma \ref{tosemi} and Proposition \ref{weyl-conj}, there exist $h_i \in H_x \cap \JJ_{\tw_x}$ such that $$g \t\i s_i I = g \t\i \t_i y_i\i I \sim_{\l, b} g \t\i \t_i h_i I.$$ In particular, by (b) we have $g I \sim_{\l, b} g h_1 I \sim_{\l, b} g \t\i I \sim_{\l, b} g \t\i \t_2 h_2 I$, that is, $\t\i \t_2 h_2 \in Q$. It follows from (b) and (a) that $\t\i \t_2 = s' s  \in Q$ and $s' \in Q$ as desired.

\

Case(2): $\s^n$ acts nontrivially on $\Psi_\b \cap J_0$. By Lemma \ref{weak} (4), $\Psi = \Phi$ and $\mu_x |_{\Psi \setminus \Psi_\b} = 0$. So we can assume that $n= d =1$, $\s$ is of order $2$, and $\Phi$ is of type $E_6$. Then $w_x = s_{\a_4} s_{\a_2}$, and it suffices to show $s, s' \in Q$, where $s = s_{\a_1} s_{\a_6}$ and $s' =  s_{\a_1 + 1} s_{\a_6 + 1}$ are all the simple affine reflections of $W_J^a \cap \JJ_{\tw_x}$.

Let $\th_0 = \a_2 + \a_4 + \a_5 + \a_6$, $\th_1 = \a_2 + \a_4 + \a_5$, $\eta_i = (w_x\s)\i(\th_i) $ and $\vartheta_i = \eta_i + \th_i$. Define $\textsl{g}_i = \textsl{g}_{g, -\th_i-1, \tw_x, 2}$ for $g \in \JJ_{b, \tw_x}$. As $\mu + \a_3^\vee, \mu + \a_3^\vee + \a_1^\vee \preceq \l$, we have $\tw_x s_{\s(\th_i) + 1} \in \Adm(\l)$ by Lemma \ref{orth}. Then $$\textsl{g}_i\i b \s(\textsl{g}_i) \subseteq  I U_{-\vartheta_i-1} \tw_x U_{-\s(\th_i)-1} \subseteq I \tw_x U_{-\s(\th_i)-1} I \subseteq I \Adm(\l) I,$$ which means $$g s_{\vartheta_0 + 1} s_{\eta_0} I = \textsl{g}_0(\infty) \sim_{\l, b} \textsl{g}_0(0) = g I = \textsl{g}_1(0) \sim_{\l, b} \textsl{g}_1(\infty) = g s_{\vartheta_1 + 1} s_{\eta_1} I.$$ As $\vartheta_0^\vee$ is $J$-dominant and $J$-minuscule, $s_{\vartheta_0 + 1} s_{\eta_0} = \o y_0\i$, where $\o = \vartheta_0^\vee \in \Omega_J \cap \JJ_{\tw_x}$ and $y_0 \in W_0$. Then $s_{\tilde \vartheta_1} s_{\eta_1} = s s_{\tilde \vartheta_0} s_{\eta_0}  s = s\o y_1\i$ for some $y_1 \in W_0$. By Proposition \ref{tosemi} \& \ref{weyl-conj}, there exist $h_0, h_1 \in H_x \cap \JJ_{\tw_x}$ such that $g \o h_0 I \sim_{\l, b} g I \sim_{\l, b} g s \o h_1 I$, that is, $\o h_0, s \o h_1 \in Q$, and hence \begin{align*} \tag{d} s \o h_1 h_0\i \o\i \in Q. \end{align*} As $h_0 h_1\i \in H_x \cap \JJ_{\tw_x} \subseteq I ((W_{J_0} W_{J_1}^a) \cap \JJ_{\tw_x}) I = I \{1, s\} I$ and $\o s \o\i = s'$, by (d) we have $s \o h_0 h_1\i \o\i \in Q \cap (I \{s, s s'\} I)$, which means $s \in Q$. Hence $H_x \cap \JJ_{\tw_x} \subseteq Q$, $\o \in Q$ and $s' = \o s \o\i \in Q$ as desired.
\end{proof}

\begin{cor} \label{lin}
Let $x \in \cs_{\l, b}^+$, $g \in \JJ_{b, \tw_x}$, and $y \in \tW$ such that $g I \sim_{\l, b} g y\i I$. Then we have $g I \sim_{\l, b} g \o {z'}\i \sim_{\l, b} g \o I$, where $z' \in W_0^J$ and $\o \in \Omega_J$ such that $y \in z' \o\i W_J^a$.
\end{cor}
\begin{proof}
It follows from Corollary \ref{pre-lin} and Proposition \ref{normal}.
%By Proposition \ref{tosemi}, there exists $h \in \ker(\eta_{M_J}) \cap \JJ_{\tw_x}$ such that $g y\i I \sim_{\l, b} g h \o {z'}\i I \sim_{\l, b} g \o {z'}\i I$, where the last relation $\sim_{\l, b}$ follows from Proposition \ref{normal}. Let $x' \in \cs_{\l, b}^+$ such that $\tw_{x'} = \o\i \tw_x \s(\o)$.  Applying Proposition \ref{weyl-conj} and Proposition \ref{normal}, there exists $h' \in H_{x'} \cap \JJ_{\tw_{x'}}$ such that $g \o {z'}\i I \sim_{\l, b} g \o h' I \sim_{\l, b} g \o I$, where the last relation $\sim_{\l, b}$ again follows from Proposition \ref{normal}. This concludes the proof.
\end{proof}

\subsection{} \label{subsec-weak}
To prove Lemma \ref{weak}, we start with a general lemma on root systems.
\begin{lem} \label{choice}
Let $\mu \in Y$, $\l \in Y^+$ and $\a \in \Phi^+$ such that $\mu \preceq \l$, $\mu + \a^\vee \le \l$, and $\mu + \a^\vee \npreceq \l$. Then there exists $\b \in \Phi^+$ such that $\<\b, \mu + \a^\vee\> \le -2$, and either $\mu + \b^\vee \preceq \l$ or $\mu + \a^\vee + \b^\vee \le \l$.
\end{lem}
\begin{proof}
We argue by induction on $\mu + \a^\vee$ via the partial order $\le$. If $\mu + \a^\vee \in Y^+$, then $\mu + \a^\vee \preceq \l$, contradicting our assumption. So there exists $\b \in \SS_0$ such that $\<\b, \mu + \a^\vee\> \le -1$ and hence $\mu + \a^\vee + \b^\vee \le \l$ (by \cite[Proposition 2.2]{Ga}). If $\<\b, \mu + \a^\vee\> \le -2$, the statement follows. Assume $\<\b, \mu + \a^\vee\> = -1$. Then $\mu + \a^\vee < s_\b(\mu + \a^\vee) \npreceq \l$. If $\b = \a$, then $\<\a, \mu\> = -3$ and $\mu + \a^\vee \preceq \mu \preceq \l$, a contradiction. So $\b \neq \a$ and $s_\b(\a) \in \Phi^+$. By induction hypothesis, for the pair $(s_\b(\mu), s_\b(\a))$ there exists $\g \in \Phi^+$ such that $$\<\g, s_\b(\mu+\a^\vee)\> = \<s_\b(\g), \mu + \a^\vee\> \le -2,$$ (which means $\b \neq \g$ and $s_\b(\g) \in \Phi^+$), and either $s_\b(\mu) + \g^\vee \preceq \l$ or $s_\b(\mu + \a^\vee) + \g^\vee \le \l$. If the former case occurs, we have $\mu + s_\b(\g^\vee) \preceq \l$, and the statement follows. Otherwise, $\<s_\b(\g), \mu\> \ge 0$ and the latter case occurs. In particular, $\<s_\b(\g), \a^\vee\> \le -2$, and hence means $\g$ is a long root. So we have $$\mu + \a^\vee + s_\b(\g^\vee) \le \mu + \a^\vee + \g^\vee + \b^\vee = s_\b(\mu+ \a^\vee) + \g^\vee \le \l,$$ and the statement also follows.
\end{proof}

\begin{proof}[Proof of Lemma \ref{weak}]
By \cite[Lemma 3.3]{N2}, there exists $x \in \cs_{\l, b}^+$ such that $\mu_x$ is weakly dominant, that is, $\<\d, \mu_x\> \ge -1$ for $\d \in \Phi^+$. As $(\l, b)$ is Hodge-Newton irreducible, there exists $\a \in K$ such that $\mu_x + \a^\vee \le \l$. We show that \begin{align*} \tag{a} & \text{(a1) there exists } \xi \in \Phi^+ \setminus \Phi_J \text{ such that } \<\a, \xi^\vee\> \le -1, \mu + \xi^\vee \preceq \l; \\ & \text{(a2) if, moreover, $\Phi$ is simply-laced, then } \<\xi, \mu_x\> = -1 \text{ and } \b \in \Phi^+ \setminus \Phi_J. \end{align*} By assumption, $\mu_x + \a^\vee \npreceq \l$. By Lemma \ref{choice}, there exists $\z \in \Phi^+$ such that $\<\z, \mu_x + \a^\vee\> \le -2$, and either $\mu_x + \z^\vee \preceq \l$ or $\mu_x + \a^\vee + \z^\vee \le \l$. As $\mu_x$ is weakly dominant, we have either (i) $\<\z, \a^\vee\> \le \<\z, \mu_x\> = -1$ or (ii) $\<\z, \a^\vee\> \le -2$ and $\<\z, \mu_x\> = 0$ or (iii) $\<\z, \a^\vee\> = -3$ and $\<\z, \mu_x\> = 1$. Take $\xi = \z$ if (i) occurs. Assume (ii) or (iii) occurs. Then $\Phi$ is non-simply-laced and $\<\a, \z^\vee\> = -1$. If $\mu_x + \z^\vee \preceq \l$, take $\xi = \z$. Otherwise, $\mu_x + \z^\vee \le \l$ is not weakly dominant (by \cite[Proposition 2.2]{Ga}). So there exists $\g \in \Phi^+$ such that $\<\g, \mu_x + \z^\vee\> \le -2$, which means $\<\g, \z^\vee\> = \<\g, \mu_x\> = -1$ since $\mu_x$ is weakly dominant and $\z$ is a long root. Then $\g \in \Phi^+ \setminus \Phi_J$ and $\mu_x + \g^\vee \preceq \l$. Note that $\a$ is a short root and $\<\a, \mu_x\> = 0$. If $\<\a, \g^\vee\> = -1$, we take $\xi = \g$. If $\<\a, \g^\vee\> = 0$, then (ii) occurs (since if (iii) occurs, then $\g = - 3\a -2\z$, contradicting that $\<\g, \mu_x\> = -1$), which means $\mu_x + \g^\vee + \z^\vee \preceq \l$. So we take $\xi = s_\g(\z)$. If $\<\a, \g^\vee\> = 1$, we take $\xi = s_\a(\g)$. It remains to show $\xi \in \Phi^+ \setminus \Phi_J^+$. Otherwise, $\xi \in \Phi_K$ since $\<\a, \xi^\vee\> \neq 0$, contradicting our assumption that $\mu + \xi^\vee \npreceq \l$. So (a) is proved.

Let $\b$ be the $J$-anti-dominant conjugate of $\xi$ under $W_J$. By (a) we have \begin{align*}\text {(b) $\<\b, \mu_x\> = -1$ if $\Phi$ is simply-laced; (c) $\mu_x + \b^\vee \preceq \l$; (d) $\b^\vee$ is non-central on $K_0$,} \end{align*} where $K_0 \subseteq \Psi_\b$ is the connected component of $K$ containing $\a$. We show that \begin{align*} \tag{e} \text{ $\b^\vee$ is $K$-minuscule.} \end{align*} Otherwise, $\<\th, \b^\vee\> \le -2$ for some $\th \in \Phi_K^+$. Then $\mu_x + \b^\vee + \th^\vee \preceq \l$. If $\<\b, \mu_x\> \ge 0$, then $\<\b, \mu_x + \b^\vee + \th^\vee\> \ge 1$ and $\mu_x + \th^\vee \preceq \l$, contradicting our assumption. Otherwise, $\<\b, \mu_x\> = -1$ and $\<s_\b(\th), \mu_x\> = -\<\th, \b^\vee\> \<\b, \mu_x\> \le -2$, contradicting that $\mu_x$ is weakly dominant. So (e) follows.

Applying \cite[Lemma 6.6]{N2} we can assume furthermore that $\b^\vee$ is $J$-anti-dominant and $J$-minuscule. Hence (1) is proved.

If $\<w_x \s^i(\b), \mu_x\> \ge 1$ for some $i \in \ZZ \setminus n\ZZ$, then $\mu_1 := \mu_x + \b^\vee - w_x \s^i(\b)^\vee \preceq \l$, which means $x_1 := x + \b^\vee - \s^i(\b)^\vee \in \cs_{\l, b}^+$. By (e), $\mu_1$ is non-central on $K_0$. As $\mu_{x_1}, \mu_1$ are conjugate by $W_J$ (see Lemma \ref{orth}), $\mu_{x_1}$ is also non-central on $K_0$, contradicting that $K_0 \subseteq J_0$. So $\<w_x \s^i(\b), \mu_x\> \le 0$ for $i \in \ZZ \setminus n\ZZ$. If $\<\s^i(\b), \mu_x\> \le -1$ for some $i \in \ZZ \setminus n\ZZ$, by Lemma \ref{pr} there exists $j \in n\ZZ$ such that $\<w_x \s^j(\b), \mu_x\> \ge 1$. Then $\mu_2 := \mu_x - w_x\s^j(\b)^\vee + \s^i(\b)^\vee \preceq \l$ and hence $x_2 := x - \s^j(\b)^\vee + \s^i(\b)^\vee \in \cs_{\l, b}^+$, which is also impossible since $\mu_2$ is non-central on $\s^j(K_0)$. So $\<\s^i(\b), \mu_x\> = \<w_x\s^i(\b), \mu_x\> = 0$ for $i \in \ZZ \setminus n\ZZ$ and (2) is proved.

If $\s^{2n}(\b) \neq \b$, then $\Phi = \Psi$ and $\Psi_\b$ is of type $D_4$, whose simple roots are $\b, \s^n(\b), \s^{2n}(\b), \a$ with $\s^n(\a) = \a$. Moreover, $J = J_0 = \co_\a$. By (2), we have $\mu_x |_{\Psi \setminus \Psi_\b} = 0$. Hence $\sum_{i = 0}^n \<\s^i(\b), \mu_x\> \ge 1$ by Lemma \ref{pr}. If $\<\s^n(\b), \mu_x\> \ge 1$, then (3) follows. If $\<\s^n(\b), \mu_x\> \le -1$, it follows by replacing $\b$ with $\s^n(\b)$. If $\<\s^n(\b), \mu_x\> = 0$, it follows by replacing $x$ with $x - \s^{2n}(\b)^\vee + \s^n(\b)^\vee \in \cs_{\l, b}^+$.

Now we assume $\s^{2n}(\b) = \b$. By (2) and Lemma \ref{pr}, \begin{align*} \tag{f} \<\b + \s^n(\b), \pr_J(\mu_x)\> = \<\b + \s^n(\b), \pr_{J_1}(\mu_x)\> > 0. \end{align*} So (3) follows if $\b = \s^n(\b)$. Assume $\b \neq \s^n(\b)$. Then $\Phi$ is simply-laced, and hence $\<\b, \mu_x\> = -1$ by (b). Moreover, $\co_\b \cup J$ is a set of simple roots of $\Psi$ by \cite[Proposition 4.2.11]{CKV}. As $\b$ is a neighbor of $K_0$ in $\Psi_\b$ and $\<\b, \mu_x\> = -1$, one checks (on the type of $\Psi_\b$) that $\<\b, \pr_{J_1}(\mu_x)\> < 0$. By (f) we have $\<w_x \s^n(\b), \mu_x\> \ge 1$ and (3) follows.

Assume $\s^n$ does not act trivially on $\Psi_\b \cap J_0$. Then $\Phi$ is simply-laced and $\<\b, \mu_x\> = -1$. We may assume $\s^n$ does not fix each point of $K_0$. Let $\a \in K_0$ such that $\<\b, \a^\vee\> = -1$. If $\s^n(\b) = \b$, then one checks directly (on the type of $\Psi_\b$ and using the assumption on $K_0$) that $\<\b, \pr_J(\mu_x)\> < 0$, which contradicts (f). So $\b \neq \s^n(\b) \in \Psi_\b$. Let $x_3 = x + \b^\vee - \s^n(\b)^\vee \in \pi_1(M_J)$. If $\b, \s^n(\b)$ are in distinct connected components of $\co_\b \cup J \setminus \{\a, \s^n(\a)\} \supseteq \co_\b \cup J_1$, then $x_3 \in \cs_{\l, b}^+$ by (2) that $\<w_x \s^n(\b), \mu_x\> \ge 1$. As $\<\a, \mu_{x_3}\> = 0$, we deduce that $\a = \s^n(\a)$ is the common neighbor of $\b, \s^n(\b)$ in $\Psi_\b$, which implies that $\s^n$ fixes each point of $K_0$, contradicting our assumption. So $\b, \s^n(\b)$ are connected in $\co_\b \cup J \setminus \{\a, \s^n(\a)\}$. Then $\a \neq \s^n(\a)$, and it follows from (f) that either $\<w_x\s^n(\b), \mu_x\> \ge 2$ or the case in (4) occurs. The former case does not occur since $x_3 \in \cs_{\l, b}^+$ but $\mu_{x_3}$ is non-central on $K_0$. So (4) follows.
\end{proof}

\section{Proof of Proposition \ref{ker-G}} \label{sec-ker}
In this section, we assume that $(\l, b)$ is Hodge-Newton irreducible. Let $\co$ be a $\s$-orbit of $\Phi^+$. We set $$\o_\co = \sum_{\a \in \co} \a^\vee \in \pi_1(M_J)^\s \cong \Omega_J^\s \subseteq \Omega_J \cap \JJ_b.$$ Let $\Psi = \Phi \cap \ZZ(\co \cup J)$. We say $\co$ is of type I (resp. type II, resp. type III) if $|\co|$ equals $n$ (resp. $2n$, resp. $3n$). Here $n \in \{d, 2d, 3d\}$ is the minimal positive integer such that $\a, \s^n(\a)$ are in the same connected component of $\Psi$ for some/any $\a \in \co$. If $\co$ is of type II or III, then $n = d$, $\Phi$ is simply-laced, and $\co \cap J$ is a set of simple roots for $\Psi$. In this case, for $\a \in \co$ we denote by $\vartheta_\a \in \Phi^+$ the sum of simple roots in the (unique) minimal $\s^n$-stable connected subset of $\co \cup J$ which contains $\a$, see \cite[\S 4.7]{CKV}.

For $x \in \cs_{\l, b}^+$ we define $$C_{\l, b, x} = \{\a \in \Phi^+ \setminus \Phi_J; \mu_x + \a^\vee \preceq \l, \a^\vee \text{ is $J$-anti-dominant and strongly $J$-minuscule} \}.$$

\subsection{} Fix a $\s$-orbit $\co$ of roots in $\Phi^+ \setminus \Phi_J$ which are $J$-anti-dominant and $J$-minuscule.
\begin{lem} \label{link}
Assume $x \overset {(\g, r)} \to x'$ with $x' = x - \g^\vee + \s^r(\g)^\vee \in \cs_{\l, b}^+$ for some $\g \in \co$ and $1 \le r \le n$. Let $\o = \g^\vee + \cdots + \s^{r-1}(\g)^\vee \in \pi_1(M_J) \cong \Omega_J$. Then for $g \in \JJ_{b, \tw_x}$ we have $$g I \sim_{\l, b} g y\i I \sim_{\l, b} g \o I \text{ for some } y \in W_0^J \o\i W_J^a \text{ if } U_{-w_J(\g)-1} \tw_x U_{-w_J \s^r(\g)-1} \subseteq I \Adm(\l) I.$$ In particular, if $\co = \co_\a$ for some $\a \in C_{\l, b, x}$, the above inclusion holds if (1) $1 \le r \le n-1$; or (2) $x = x'$; or (3) $\mu_x + \vartheta_\g^\vee \npreceq \l$ when $\co$ is of type II and $r=n$.
\end{lem}
\begin{proof}
Let $\tilde \th = w_J\s^{r-1}(\g) + 1 \in \tPhi^+$. Suppose \begin{align*} \tag{a} U_{-\s^{1-r}(\tilde \th)} \tw_x U_{-\s(\tilde \th)} = U_{-w_J(\g)-1} \tw_x U_{-w_J \s^r(\g)-1} \subseteq I \Adm(\l) I.\end{align*} By \cite[Lemma 6.5]{N2} we can assume that \begin{align*} \tag{b} x \overset {(\g, r)} \rightarrowtail x', \text{ and hence }  (\tw_x\s)^i(\tilde \th) = \s^i(\tilde \th) \text{ for } 1-r \le i \le 0. \end{align*} Define $\textsl{g} = \textsl{g}_{g, -\tilde \th, \tw_x, r}$ for $g \in \JJ_{b, \tw_x}$. By (a) and (b) we have $$\textsl{g}\i b \s(\textsl{g}) \subseteq U_{-\s^{1-r}(\tilde \th)} \tw_x U_{-\s(\tilde \th)} \subseteq I \Adm(\l) I,$$ which means $$g I = \textsl{g}(0) \sim_{\l, b} \textsl{g}(\infty) = g s_{\tilde \th} \cdots s_{\s^{r-1}(\tilde \th)} I = g \o u\i I$$ for some $u \in W_0$ as desired. The relation $g I \sim_{\l, b} g \o I$ follows from Corollary \ref{lin}.

If $\co$ is of type II and $r=n$, then $\vartheta_\g^\vee$ is $J$-anti-dominant and $J$-minuscule, which means $\mu_x + \vartheta_\g^\vee$ is $J$-minuscule and hence $\mu_x + \vartheta_\g^\vee \preceq \mu_x + (w_J(\g) + w_x w_J \s^r(\g))^\vee$. Thus the `` Moreover" part follows from Lemma \ref{line} (*) by noticing that $\<w_J(\g), w_x w_J \s^r(\g^\vee)\> = 0$ if $1 \le r \le d-1$.
\end{proof}

Let $\ca_{\l, b}$ be the group of elements $\o \in \pi_1(M_J)^\s \cong \Omega_J^\s$ which fix some/any connected component of $X(\l, b)$.
\begin{lem} \label{type-I}
Suppose $\co = \co_\xi$ for some $\xi \in C_{\l, b, x}$. If $\co$ is of type I, then there exist $\g \in \co$, $1 \le r \le n$, and $x' \in \cs_{\l, b}^+$ such that $x \overset {(\g, r)} \to x'$. Moreover, $\o_\co \in \ca_{\l, b}$.
\end{lem}
\begin{proof}
Note that $\mu + \a^\vee \preceq \l$. If $\<w_J \s^r(\a), \mu_x\> \ge 1$ for some $1 \le r \le n-1$, then $\<w_J \s^r(\a), \a^\vee\> = 0$, which means $x \overset {(\s^r(\a), n-r)} \to x' \overset {(\a, r)} \to x$ with $x' = x -\s^r(\a^\vee) + \a^\vee \in \cs_{\l, b}^+$. Otherwise, $\<w_J\s^i(\a), \mu_x\> \le 0$ for $1 \le i \le n-1$, which means $\<w_J(\a), \mu_x\> \ge 1$ by Lemma \ref{pr}. So $x \overset {(\a, n)} \to x$ and the first statement follows. As $\co$ is of type I, the second statement follows from Corollary \ref{lin} and Lemma \ref{link} (1) (resp. Lemma \ref{link} (2)) if $r \neq n$ (resp. $r = n$).
\end{proof}

\begin{lem} \label{type-II-1}
Suppose $\co$ is of type II. Assume $\mu_{x''} + \vartheta_\b^\vee \npreceq \l$ for any $x'' \in \cs_{\l, b}^+$ and $\b \in \co$. If there exist $\g \in \co$, $n+1 \le r \le 2n-1$, and $x' \in \cs_{\l, b}^+$ such that $x \overset {(\g, r)} \rightarrowtail x'$, then

(1) $\<\s^i(\g), \mu_x\> = 0$, $w_x \s^i(\g) = \s^i(\g)$ for $1 \le i \neq r-n \le r-1$;

(2) $w_x \s^{r-n}(\g) = \s^{r-n}(\vartheta_\g - \s^n(\g))$ and $\<w_x \s^{r-n}(\g), \mu_x\> = 1$;

(3) $\<w_x(\vartheta_\g - \s^n(\g)), \mu_x\> \ge 1$.

Moreover, $g I \sim_{\l, b} g \o I$ for $g \in \JJ_{b, \tw_x}$, where $\o = \g^\vee + \cdots \s^{r-1}(\g^\vee) \in \pi_1(M_J) \cong \Omega_J$.
\end{lem}
\begin{proof}
Write $x' = x + \s^r(\g^\vee) - \s^{-r}(\s^r(\g^\vee))$. Then (1), (2) and (3) follow from \cite[Lemma 8.2]{N2} by using $\s\i$ instead of $\s$. Let $\tilde \th = w_J\s^{r-1}(\g) + 1 \in \tPhi^+$ and $\tilde \vartheta_\g = \vartheta_\g + 1 \in \tPhi^+$. By (1) and (2) we have $(\tw_x \s)^i(\tilde \th) = \s^i(\tilde \th) = w_J \s^{i+r-1}(\g) + 1$ for $1-n \le i \le 0$, and $$(\tw_x \s)^i(\tilde \th) = \s^{i+n-1} w_x\i w_J \s^{r-n}(\g) = \s^{i+n-1} w_J w_x \s^{r-n}(\g) = w_J \s^{i+r-1}(\vartheta_\g - \s^n(\g))$$ for $1-r \le i \le -n$. Define $\textsl{g} = \textsl{g}_{g, -\tilde\th, \tw_x, r}$ for $g \in \JJ_{b, \tw_x}$. Then we have $$\textsl{g}\i b \s(\textsl{g}) \subseteq I U_{-w_J(\tilde \vartheta_\g)} \tw_x U_{-\s(\tilde \th)} I \subseteq I \tw_x U_{-\s(\tilde \th)} I \subseteq I \Adm(\l) I,$$ where the second inclusion follows from (1) and (3) that $\<w_x(\vartheta_\g), \mu_x\> \ge 1$. Thus $$g I = \textsl{g}(0) \sim_{\l, b} \textsl{g}(\infty) = g s_{(\tw_x\s)^{1-r}(\tilde\th)} \cdots s_{(\tw_x\s)\i(\tilde\th)} s_{\tilde \th} I = g \o u\i I,$$ where $u \in W_0$ and $\o = \g^\vee + \cdots + \s^{r-1}(\g^\vee) \in \pi_1(M_J) \cong \Omega_J$. By Corollary \ref{lin} we have $g I \sim_{\l, b} g \o I$ as desired.
\end{proof}

\begin{lem} \label{type-II-2}
Suppose $\co = \co_\xi$ for some $\xi \in C_{\l, b, x}$ and $\co$ is of type II. Assume $\mu_{x''} + \vartheta_\b^\vee \npreceq \l$ for any $x'' \in \cs_{\l, b}^+$ and $\b \in \co$. If there do not exist $\g \in \co$, $1 \le r \le 2n-1$, and $x' \in \cs_{\l, b}^+$ such that $x \overset {(\g, r)} \rightarrowtail x'$. Then there exists $\a \in \co$ such that

(1) $\<\s^i(\a), \mu_x\> = 0$, $w_x \s^i(\a) = \s^i(\a)$ for $1 \le i \neq n \le 2n-1$;

(2) $w_x \s^n(\a) = \vartheta_\a - \a$ and $\<w_J \s^n(\a), \mu_x\> = 1$;

(3) $\<w_x(\vartheta_\a), \mu_x + \a^\vee\> \ge 1$;

(4) $\<w_x(\vartheta_\a), \mu_x\> \ge 1$.

As a consequence, $\o_\co \in \ca_{\l, b}$.
\end{lem}
\begin{proof}
The statements (1), (2) and (3) follow from \cite[Lemma 8.3 \& Lemma 8.4]{N2}. Note that $\vartheta_\a$ is $J$-anti-dominant. So (4) follows from (1) and Lemma \ref{pr}. By (3) and (4) we have \begin{align*} \tag{a} \<w_x(\vartheta_\a), \mu_x\> \ge 1, \text{ and either } \<w_x(\vartheta_\a), \mu_x\> \ge 2  \text{ or } \<w_x(\vartheta_\a), \a^\vee\> \ge 0. \end{align*} Let $g \in \JJ_{b, \tw_x}$ and $\tilde \th = w_J\s\i(\a) + 1 \in \tPhi^+$, and $\tilde \vartheta = w_J \s\i(\vartheta_\a) + 1 \in \tPhi^+$. By (1) and (2) we have $(\tw_x\s)^{1-n}(\tilde \vartheta) = \s^{1-n}(\tilde \vartheta) = w_J(\vartheta_\a)$ and $$(\tw_x\s)^{-n}(\tilde \th) = \s\i w_x\i w_J \s^{-n}(\a) = \s\i w_J w_x \s^{-n}(\a) = w_J \s\i(\vartheta_\a - \a).$$

Define $\textsl{g}: \PP^1 \to G(\brF) / I$ by $$\textsl{g}(z) = g U_{-\tilde \th}(z) \cdots {}^{(\tw_x \s)^{1-n}} U_{-\tilde \th}(z)  U_{-\tilde \vartheta}(c z^{1 + q^{-n}}) \cdots {}^{(\tw_x \s)^{1-n}} U_{-\tilde \vartheta}(c z^{1 + q^{-n}}) I,$$ where $c \in \co_\brF^\times$ (as $\Phi$ is simply-laced) such that $${}^{(\tw_x \s)^{-n}} U_{-\tilde \th}(z) U_{-\tilde \th}(z) U_{-\tilde \vartheta}(c z^{1 + q^{-n}}) = U_{-\tilde \th}(z) {}^{(\tw_x \s)^{-n}} U_{-\tilde \th}(z).$$

Then by (1) we compute that $$\textsl{g}\i b \s(\textsl{g}) =  U_{-w_J(\tilde \vartheta_\a)} \tw_x U_{-\s(\tilde \th)} I \subseteq I \tw_x U_{-\s(\tilde \th)} I \subseteq I \Adm(\l)\s I,$$ where the first inclusion follows by (a) that ${}^{\tw_x\i} U_{-w_J(\tilde \vartheta_\a)}, [{}^{\tw_x\i} U_{-w_J(\tilde \vartheta_\a)}, U_{-\s(\tilde \th)}] \subseteq I$. Thus $$g I = \textsl{g}(0) \sim_{\l, b} \textsl{g}(\infty) = g (s_{\tilde \vartheta} s_{\th'}) \cdots \s^{1-n}( s_{\tilde \vartheta} s_{\th'}) I = g \o_\co u\i I,$$ where $\th' = (\tw_x\s)^{1-n}(\tilde \vartheta) \in \Phi$ and $u \in W_0$. By Corollary \ref{lin} we have $g I \sim_{\l, b} g \o_\co I$ and $\o_\co \in \ca_{\l, b}$ as desired.
\end{proof}

\subsection{} Now we have the following result.
\begin{prop} \label{generator}
Let $\co$ be the $\s$-orbit of some element in $\cup_{x \in \cs_{\l, b}^+} C_{\l, b, x}$. Then $\o_\co \in \ca_{\l, b}$.
\end{prop}
\begin{proof}
If $\co$ is of type I, the statement follows from Lemma \ref{type-I}. If $\mu_{x''} + \vartheta_\b^\vee \preceq \l$ for some $x'' \in \cs_{\l, b}^+$ and $\b \in \co$, then we also have $\o_\co = \o_{\co_{\vartheta_\b}} \in \ca_{\l, b}$ since  $\co_{\vartheta_\b}$ is of type I. Assume $\mu_{x''} + \vartheta_\b^\vee \npreceq \l$ for any $x'' \in \cs_{\l, b}^+$ and $\b \in \co$. If $\co$ is of type III, the statement is proved in \S \ref{subsec-III}. Suppose $\co$ is of type II. By Lemma \ref{type-II-2} we can assume that there exist $\g \in \co$, $1 \le r \le 2n-1$, and $x' \in \cs_{\l, b}^+$ such that $x \overset {(\g, r)} \to x'$, and hence $x' \overset {(\s^r(\g), 2n-r)} \to x$. If $n+1 \le r \le 2n-1$ (resp. $1 \le r \le n$), we have $g I \sim_{\l, b} g \o I$ by Lemma \ref{type-II-1} (resp. by Lemma \ref{link} (1) \& (3)), where $\o = \g^\vee + \cdots \s^{r-1}(\g^\vee) \in \pi_1(M_J) \cong \Omega_J$. Similarly, we have $g \o \sim_{\l, b} g \o \o' I = g \o_\co I$, where $\o' = \s^r(\g^\vee) + \cdots + \s^{2n-1}(\g^\vee) \in \pi_1(M_J) \cong \Omega_J$. So $g I \sim_{\l, b} g \o_\co I$ and $\o_\co \in \ca_{\l, b}$ as desired.
\end{proof}

\begin{proof}[Proof of Proposition \ref{ker-G}]
First note that $(\ZZ\Phi^\vee / \ZZ\Phi_J^\vee)^\s$ is panned by $\o_\co$, where $\co$ ranges over $\s$-orbits of $\SS_0$. Let $J \subseteq \SS_0' \subseteq \SS_0$ be such that $\o_\co \in \ca_{\l, b}$ for each $\s$-orbit of $\SS_0'$. It suffices to show $\SS_0' = \SS_0$. Assume otherwise. Following the proof of \cite[Proposition 4.3]{N2}, we can assume that $\Phi$ is simply-laced, and there exist $\a = \s^d(\a) \in \SS_0 \setminus \SS_0'$, $\vartheta = \s^d(\vartheta) \in \Phi^+$ such that $\vartheta^\vee - \a^\vee \in \ZZ \Phi_{\SS_0'}^\vee$ and either (b1) $\vartheta \in \cup_{x \in \cs_{\l, b}^+} C_{\l, b, x}$, or (b2) $x \overset {(\b, d)} \to x'$ and $x \overset {(\vartheta + \b, d)} \to x'$ for some $x \in \cs_{\l, b}^+$ and $\b \in \Phi_{\SS_0'} \setminus \Phi_J$ such that $x' = x - \b^\vee + \s^d(\b^\vee) \in \cs_{\l, b}^+$ and $\vartheta + \b \in \Phi^+$.

Note that $|\co_\a| = |\co_{\vartheta}| = d$ and $\o_{\co_\a}\i \o_{\co_\vartheta} \in (\ZZ \Phi_{\SS_0'}^\vee / \ZZ\Phi_J^\vee)^\s \subseteq \ca_{\l, b}$. If (b1) occurs, then $\o_{\co_\vartheta} \in \ca_{\l, b}$ by Proposition \ref{generator}. Hence $\o_\co \in \ca_{\l, b}$ and $\a \in \SS_0'$, which is a contradiction. Suppose (b2) occurs. Let $\o = \b^\vee + \cdots + \s^{d-1}(\b^\vee) \in \pi_1(M_J) \cong \Omega_J$. Then $\o \o_{\co'} = (\b + \vartheta)^\vee + \cdots + \s^{d-1}((\b+\vartheta)^\vee) \in \pi_1(M_J) \cong \Omega_J$. We claim that \begin{align*} \tag{a} g \o I \sim_{\l, b} g I \sim_{\l, b} g \o \o_{\co_\vartheta} I \text{ for } g \in \JJ_{b, \tw_x}. \end{align*} Given (a) we have $g \o I \sim_{\l, b} g \o \o_{\co_\vartheta} I$, and hence $\o_{\co_\vartheta} \in \ca_{\l, b}$, which is again a contradiction. Thus $\SS_0' = \SS_0$ as desired.

It remains to show (a). By symmetry, it suffices to show $g I \sim_{\l, b} g \o I$. By switching $x$ with $x'$ we can assume $\b \in \Phi^+ \setminus \Phi_J$ and $\b$ is $J$-anti-dominant and $J$-minuscule (see \cite[Lemma 6.6]{N2}). In particular, $\s^d(\b) \in C_{\l, b, x}$. If $\co_\b$ is of type I, it follows from Lemma \ref{link}. If $\co_\b$ is of type III, it follows from Lemma \ref{small}. If $\co_\b$ is of type II, by Lemma \ref{line} and Lemma \ref{link} we have $$\text{either } g I \sim_{\l, b} g \o I \text{ or } g\o \sim_{\l, b} g \o \o' I = g \o_{\co_\b} I \text{ for } g \in \JJ_{b, \tw_x},$$ where $ \o' = \s^d(\b^\vee) + \cdots + \s^{2d-1}(\b^\vee) \in \pi_1(M_J) \cong \Omega_J$. Note that $g I \sim_{\l, b} g \o_{\co_\b} I$ by Proposition \ref{generator}. So we always have $g I \sim_{\l, b} g \o I$ as desired. So (a) is proved.
\end{proof}

\section{The case that $\s$ has order $3d$} \label{sec-order-3}
We assume that $\s$ has order $3d$. Then some/any connected component of $\SS_0$ is of type $D_4$.

\subsection{} \label{sec-case}
Let $\a, \b \in \SS_0$ such that $\<\a, \b^\vee\> = -1$ and $\b = \s^d(\b)$. Then the subset $\{\a, \s^d(\a), \s^{2d}(\a), \b\}$ is a connected component of $\SS_0$. Assume $J = \co_\b$.

Let $x, x' \in \cs_{\l, b}^+$ such that $x \overset {(\a, r)} \to x'$ for some $J$-anti-dominant root $\a \in \Phi^+ \setminus \Phi_J$ and $1 \le r \le 3d-1$. Let $\o = \g^\vee + \cdots + \s^{r-1}(\g)^\vee \in \pi_1(M_J) \cong \Omega_J$.
\begin{lem} \label{small}
If $1 \le r \le d$, then $g I \sim_{\l, b} g y\i I$ for $g \in \JJ_{b, \tw_x}$ and some $y \in W_0^J \o\i W_J^a$.
\end{lem}
\begin{proof}
As in the proof Lemma \ref{link}, we can assume $x \overset {(\a, r)} \rightarrowtail x'$, and it suffices to show $$U_{-(\a+\b)-1} \tw_x U_{-\s^r(\a+\b)-1} \subseteq I \Adm(\l) I.$$ Assume otherwise. Then $r = d$. Moreover, by Lemma \ref{line} (*) we have $\<\a+\b, w_x\s^d(\a+\b)^\vee\> = -1$ (which implies $\<\b, \mu_x\> = 1$ and $w_x\s^d(\a+\b) = s_\b(\s^d(\a)+\b) = \s^d(\a)$) and  $$\<\b, \mu_x\> = \<\a+\b, \mu_x\> = -\<\s^d(\a), \mu_x\> = 1, \text{ and } \mu_x \pm \d^\vee \preceq \l,$$ where $\d = \a + \b + \s^d(\a)$. As $\d$ is central for $J = \co_\b$, by Lemma \ref{orth} (2) we have $$U_{-(\a+\b)-1} \tw_x U_{-\s^r(\a+\b)-1} \subseteq I U_{-(\d + 1)} \tw_x I \subseteq I \{s_{\d+1}\tw_x, \tw_x\} I \subseteq I \Adm(\l) I,$$ which is a contradiction.
\end{proof}

\begin{lem} \label{large}
Suppose $2d \le r \le 3d-1$ and the following conditions hold:

(1) $\<\a, \mu_x\> \ge 1$;

(2) if $r = 2d$, then $\<\s^d(\a), \mu_x\> = 0$;

(3) if $2d+1 \le r \le 3d-1$, then $\<\s^r(\b), \mu_x\> = 1$, $\<\b, \mu_x\> = 0$, and $\<\s^i(\a), \mu_x\> = 0$ for $i \in \{r-d, r-2d, d, 2d\}$;

(4) $\tw_x \s^i(\a) = \s^i(\a)$ for $1 \le i \le r-1$ with $i \notin \{r-d, r-2d, d, 2d\}$.

Then we have $g I \sim_{\l, b} g y\i I$ for $g \in \JJ_{b, \tw_x}$ and some $y \in W_0^J \o\i W_J^a$.
\end{lem}
\begin{proof}
Let $\tilde \th = \s^{r-1}(\a + \b) + 1 \in \tPhi^+$. Define $\textsl{g} = \textsl{g}_{g, -\tilde \th, \tw_x, r}$ for $g \in \JJ_{b, \tw_x}$.

Case(1): $r = 2d$. By (2) and (4) we have \begin{align*} \textsl{g}\i b \s(\textsl{g}) \subseteq \begin{cases} I U_{-(\a + \b + \s^d(\a))-1} \tw_x U_{-\s^r(\a+\b)-1} I, & \text{ if } \<\b, \mu_x\> = 1; \\ I U_{-(\a + \b)-1} \tw_x U_{-\s^r(\a+\b)-1} I, & \text{ if } \<\b, \mu_x\> = 0; \end{cases} \end{align*} By (1) and (2), $\<\a + \b, \mu_x\> = \<\a + \b + \s^d(\a), \mu_x\> \ge \<\b, \mu_x\> + 1$, which means $$\textsl{g}\i b \s(\textsl{g}) \subseteq \tw_x U_{-\s^r(\a + \b) - 1} I \subseteq I \Adm(\l) I.$$ So $g = \textsl{g}(0) \sim_{\l, b} \textsl{g}(\infty) = g s I$, where $s = \prod_{i=0}^{d-1} s_{\s^i(\a + \b + \s^d(\a)) +1} \prod_{i=0}^{d-1} s_{\s^i(\a)}$ if $\<\b, \mu_x\> = 1$, and $s = \prod_{i=0}^{2d-1} s_{\s^i(\a + \b)+1}$ if $\<\b, \mu_x\> = 0$.

Case(2): $2d+1 \le r \le 3d-1$. Let $\vartheta = \a + \s^d(\a) + \s^{2d}(\a) + 2\b$. By (3) and (4),\begin{align*} \textsl{g}\i b \s(\textsl{g}) \subseteq I U_{-\vartheta-1} \tw_x U_{-\s^r(\a+\b)-1} I \subseteq I \tw_x U_{-\s^r(\a+\b)-1} I \subseteq I \Adm(\l) I, \end{align*} which means $g I = \textsl{g}(0) \sim_{\l, b} \textsl{g}(\infty) = g s I$, where $$s = \prod_{i=0}^{r-1} s_{\s^i(\vartheta)+1} s_{\s^i(\a+\b)} s_{\s^{i+d}(\a)} \prod_{i=r}^{d-1} s_{\s^i(\a + \b + \s^d(\a)) +1} s_{\s^i(\a + \b)}.$$ The proof is finished.
\end{proof}

The following two lemmas follow from the same construction in Lemma \ref{large}.
\begin{lem} \label{good-case}
Assume $d+1 \le r \le 2d-1$ and the following conditions hold:

(1) $\<\b, \mu_x\> = 0$ and $\<\s^r(\b), \mu_x\> \in \{0, 1\}$;

(2) $\<\s^d(\a), \mu_x\> = \<\s^{r-d}(\a), \mu_x\> = 0$, and $\<\a, \mu_x\> \ge 1$;

(3) $\tw_x\s^i(\a) = \s^i(\a)$ for $1 \le i \le r-1$ with $i \notin \{r-d, d\}$.

Then we have $g I \sim_{\l, b} g y\i I$ for $g \in \JJ_{b, \tw_x}$ and some $y \in W_0^J \o W_J^a$.
\end{lem}

\begin{lem}[{\cite[Lemma 8.6]{N2}}] \label{full}
Suppose $\<\b, \mu_x\> = 1$, $\<\s^d(\a), \mu_x\> = \<\s^{2d}(\a), \mu_x\> = 0$, $\<\a, \mu_x\> \ge -1$, and $\tw_x \s^i(\a) = \a$ for $i \in \ZZ \setminus d\ZZ$. Then $g I \sim_{\l, b} g y\i I$ for $g \in \JJ_{b, \tw_x}$ and some $y \in W_0 \o_{\co_\a}\i W_J^a$. Here $\o_{\co_\a} = \a^\vee + \cdots \s^{3d-1}(\a^\vee) \in \pi_1(M_J) \cong \Omega_J$.
\end{lem}

\begin{lem} \label{central}
Let $x_1, x_2 \in \cs_{\l, b}^+$, $\d = \a + \b + \s^{2d}(\a)$ and $1 \le k \le 3d-1$ such that $x_1 \overset {(\d, k)} \to x_2$. Then we have $g I \sim_{\l, b} g y\i I$ for $g \in \JJ_{b, \tw_{x_1}}$ and some $y \in W_0^J \o\i W_J^a$. Here $\o = \d^\vee + \cdots \s^{k-1}(\d^\vee) \in \pi_1(M_J) \cong \Omega_J$.
\end{lem}
\begin{proof}
It follows from Lemma \ref{type-I} by noticing that $\co_\d$ is of type I.
\end{proof}

\begin{lem} \label{bad-case}
Assume $d+1 \le r \le 2d-1$ and the following conditions hold:

(1) $\<\b, \mu_x\> = 1$ and $\<\s^r(\b), \mu_x\> = 0$;

(2) $\<\s^d(\a), \mu_x\> = -1$, $\<\s^{r-d}(\a), \mu_x\> = 0$, $\<\a, \mu_x\> \le 0$, and $\<\s^r(\a), \mu_x\> \le -1$;

(3) $\tw_x\s^i(\a) = \s^i(\a)$ for $1 \le i \le r-1$ with $i \notin \{r-d, d\}$

Then we have $\JJ_{b, \tw_x} \sim_{\l, b} \JJ_{b, \tw_{x'}}$.
\end{lem}
\begin{proof}
Let $\d = \a + \b + \s^{2d}(\a)$. Assume $\mu_x - \d^\vee \preceq \l$. By (2) we have $$x \overset {(\d, r)} \to x'' := x - \d^\vee + \s^r(\d^\vee) \overset {(\s^{r-d}(\a), 3d-r)} \to x'.$$ So $\JJ_{b, \tw_x} \sim_{\l, b} \JJ_{b, \tw_{x''}}$ by Lemma \ref{central}. It suffices to show $\JJ_{b, \tw_{x''}} \sim_{\l, b} \JJ_{b, \tw_{x'}}$. If $\<\s^r(\s), \mu_{x''}\> \le -1$, then $$x'' \overset {(\s^{r-d}, d)} \to  x'' - \s^{r-d}(\a^\vee) + \s^r(\a^\vee) \overset {(\s^r(\a), 2d-r)} \to x',$$ and the statement follows from Lemma \ref{small} that $\JJ_{b, \tw_{x''}} \sim_{\l, b} \JJ_{b, \tw_{x'}}$. Otherwise, by (2) we have $\<\s^r(\a), \mu_x\> = -1$, that is, $\<\s^r(\s), \mu_{x''}\> = 0$. the statement follows from Lemma \ref{good-case} that $\JJ_{b, \tw_{x''}} \sim_{\l, b} \JJ_{b, \tw_{x'}}$. Let $l = \min\{r+1 \le i \le 2d-1; \<\s^i(\a), \mu_x\> \neq 0\}$. If $\<\s^l(\a), \mu_x\> \ge 1$, then $$x'' \overset{(\s^l(\a), 2d-1)} \to x'' - \s^l(\a^\vee) + \s^{2d}(\a^\vee) \overset{(\s^{r-d}(\a), l+d-r)} \to x',$$ and the statement follows from Lemma \ref{small} \& \ref{good-case}. If $\<\s^l(\a), \mu_x\> \le -1$, then $$x'' \overset{(\s^{r-d}(\a), k+d-r)} \to x'' - \s^{r-d}(\a^\vee) + \s^l(\a^\vee) \overset{(\s^l(\a), 2d-l)} \to x',$$ and the statement also follows from Lemma \ref{small} \& \ref{good-case}.

Now we assume $\mu_x - \d^\vee \npreceq \l$, which means (as $\mu_x - \a^\vee - \b^\vee = \mu_{x - \a^\vee} \preceq \l$) that \begin{align*} \tag{a} \<\s^{2d}(\a), \mu_x\> \le -1. \end{align*} If $\<\s^{r+d}(\a), \mu_x\> \ge 1$, then we have $$x \overset {(\s^{r-d}(\d), d)} \to x - \s^{r+d}(\a)^\vee + \s^r(\a)^\vee  \overset {(\s^{r+d}(\a), 2d-r)} \leftarrow x',$$ and the statement follows from Lemma \ref{central} and Lemma \ref{small}. So we assume \begin{align*} \tag{b} \<\s^{r+d}(\a), \mu_x\> \le 0.  \end{align*}

By (a), (b), (1), and (2), we have $$\sum_{i \in \{r-d, r, r+d, 0, d, 2d\}} \<\s^i(\a), \pr_J(\mu_x)\> < 0.$$ By Lemma \ref{pr}, there exists $r+1 \le k \le 3d-1$ with $k \notin \{2d, r+d\}$ such that \begin{align*} \tag{c} k = \min\{r+1 \le i \le 3d-1; \<\s^i(\a), \mu_x\> \ge 1\}. \end{align*}

Suppose $\<\s^j(\a), \mu_x\> \le -1$ for some $r+1 \le j \le 3d-1$ with $j \notin \{2d, k+d, k-d, r+d\}$. Let \begin{align*} z = x - \s^{k_1}(\d)^\vee + \s^{j_1}(\d)^\vee, \ z' = x' - \s^{k_1}(\d)^\vee + \s^{j_1}(\d)^\vee \in \cs_{\l, b}^+, \end{align*} where $k_1 = k + d$ if $k > 2d$ and $k_1 = k$ otherwise, and $j_1$ is defined in the same way. By Lemma \ref{central}, we have $\JJ_{b, \tw_x} \sim_{\l, b} \JJ_{b, \tw_z}$ and $\JJ_{b, \tw_{x'}} \sim_{\l, b} \JJ_{b, \tw_{z'}}$. Moreover, there exist $z_1, z_2 \in \cs_{\l, b}^+$ such that \begin{align*} & z \overset {(\a, k-2d)} \to z_1 \overset {(\s^{k-2d}(\a), 2d+r-k)} \to z'  \text{ if } r+d+1 \le k \le 3d-1; \\ & z \overset {(\a, k-d)} \to z_1 \overset {(\s^{k-2d}(\a), d+r-k)} \to z',  \text{ if } r+1 \le k \le 2d-1; \\ & z \overset {(\a, k-2d)} \to z_1 \overset {(\s^{k-d}(\a), d+r-k)} \to z_2 \overset {(\s^{k-2d}(\a), d)} \to z',  \text{ if } 2d+1 \le k \le r+d-1. \end{align*} By Lemma \ref{small}, $\JJ_{b, \tw_z} \sim_{\l, b} \JJ_{b, \tw_{z'}}$ and the statement follows. So we can assume \begin{align*} \tag{d} \<\s^i(\a), \mu_x\> = 0 \text{ for } 1 \le i \le k-1 \text{ with } i \notin \{r-d, r, r+d, d, 2d\}. \end{align*} As $\<\s^{r-d}(\a), \mu_{x'}\> = -1$, we have $y :=  x' + \s^{r-d}(\a)^\vee - \s^k(\a)^\vee \in \cs_{\l, b}^+$.

Case(1): $r+1 \le k \le 2d-1$. Then $$x \overset{(\s^k(\d), r-k)} \to x - \s^k(\d)^\vee + \s^r(\d)^\vee \overset {(\a, k-d)} \to y \overset {(\s^{r-d}(\a), k-r+d)} \to x'.$$ By Lemma \ref{small}, it suffices to show $\JJ_{b, \tw_y} \sim_{\l, b} \JJ_{b, \tw_{x'}}$. If $\<\s^r(\a), \mu_x\> \le -2$, that is, $\<\s^r(\a), \mu_y\> \le -1$, it follows from that $$y \overset {(\s^{r-d}(\a), d)} \to x' + \s^r(\a)^\vee - \s^k(\a)^\vee \overset {(\s^r(\a), k-r)} \to x'.$$ Otherwise, we have $\<\s^r(\a), \mu_x\> = -1$ by (2), that is, $\<\s^r(\a), \mu_y\> = 0$. Then the statement follows from Lemma \ref{good-case}.

Case(2): $2d+1 \le k \le 3d-1$. Then we have $$x \overset{(\s^{k+d}(\d), r-k-d)} \to x - \s^{k+d}(\d)^\vee + \s^r(\d)^\vee \overset {(\a, k-2d)} \to y \overset {(\s^{r-d}(\a), k-r+d)} \to x'.$$ Again, it suffices to show $\JJ_{b, \tw_y} \sim_{\l, b} \JJ_{b, \tw_{x'}}$. If $k \le r+d-1$, it follows similarly as in Case(1). Otherwise, it follows from that \begin{align*} & y \overset {(\s^{k-d}(\a), r+2d-k)} \to y - \s^{k-d}(\a)^\vee + \s^{r+d}(\a)^\vee \overset {(\s^{k-2d}(\a), d)} \to y - \s^{k-2d}(\a)^\vee + \s^{r+d}(\a)^\vee \\ &\ \overset{(\s^{r-d}(\a), k-r-d)} \to y - \s^{r-d}(\a)^\vee + \s^{r+d}(\a)^\vee  \overset {(\s^{r+d}(\a), k-r-d)} \to x', \end{align*} where the first arrow follows from (b) that $\<\s^{r+d}(\a), \mu_y\> = \<\s^{r+d}(\a), \mu_x\> - 1 \le -1$.

%Notice that $\<\s^i(\a), \mu_y\> \le 0$ for $i \in \{r, r+d, 2d\}$. If $\<\s^i(\a), \mu_y\> \le 0$ for $i \in \{r, r+d, 2d\}$, it follows from Lemma \ref{large}. If $\<\s^{2d}(\a), \mu_x\> \le -1$, then it follow from that $$y \overset {(\s^{k-d}(\a), 3d-k)} \to x_5 \overset {(\s^{k-2d}(\a), d)} \to x_6 \overset {(\s^{r-d}(\a), k-r-d)} \to x_7 \overset {(\s^{2d}(\a), k-2d)} \to x'.$$
\end{proof}

\subsection{} \label{subsec-III}
Now we finish the proofs for the case that $\s$ has order $3d$.

\begin{proof}[Proof of Proposition \ref{surj}] Let $x, x' \in \cs_{\l, b}^+$. To show $\JJ_{b, \tw_x} \sim_{\l, b} \JJ_{b, \tw_{x'}}$, by Proposition \ref{conneted} we can assume $x \overset {(\g, r)} \rightarrowtail x'$ for some $1 \le r \le 2d-1$ and $\g \in \Phi^+ \setminus \Phi_J$ with $\g^\vee$ is $J$-anti-dominant and $J$-minuscule. In particular, $\s^r(\g) \in C_{\l, b, x}$. If $\co_\g$ is of type I, the statement follows from Lemma \ref{link} and Corollary \ref{pre-lin}. Otherwise, we can assume $J = \co_\b$ and $\g = \a$ as in \S \ref{sec-case}. If $1 \le r \le  d$, the statement follows from Lemma \ref{small} and Corollary \ref{pre-lin}. Otherwise, by the proof of \cite[Proposition 6.8]{N2}, either Lemma \ref{good-case} or Lemma \ref{bad-case} applies. So the statement also follows.
\end{proof}

\begin{proof} [Proof of Proposition \ref{generator}]
As $\co$ is of type III, we can assume $\co = \co_\a$ and $J = \co_\b$, where $\a, \b$ are as in \S \ref{sec-case}. Again we can assume that $\mu_{x''} + \vartheta_\g^\vee \npreceq \l$ for any $x'' \in \cs_{\l, b}^+$ and $\g \in \co$. If there do not exist $\g \in \co$, $1 \le r \le 3d-1$, and $x' \in \cs_{\l, b}^+$ such that $x \overset {(\g, r)} \rightarrowtail x'$, by \cite[Lemma 8.6]{N2} the statement follows from Lemma \ref{large} and Corollary \ref{lin}. Assume otherwise. Then there exists $x_i \in \cs_{\l, b}^+$, $\g_i \in \co$ and $1 \le r_i \le 3d-1$ for $1 \le i \le m$ such that $\o_\co = \sum_{i=1}^m \sum_{j=0}^{r_i-1} \s^j(\g_i)^\vee \in \pi_1(M_J)$ and $$x = x_0 \overset {(\g_1, r_1)} \to x_1 \overset {(\g_2, r_1)} \to \cdots \overset {(\g_m, r_m)} \to x_m = x.$$ if $d+1 \le r_i \le 2d-1$, then either Lemma \ref{good-case} or Lemma \ref{bad-case} occurs. If for each $1 \le i \le m$ we have either $r_i \le d$ or $2d \le r_i \le 3d-1$ or Lemma \ref{good-case} (for $(x, x', \a, r) = (x_{i-1}, x_i, \g_i, r_i)$) occurs, it follows that $\o_\co \in \ca_{\l, b}$ by Lemma \ref{small}, \ref{large}, Lemma \ref{good-case} and Corollary \ref{lin}. Otherwise, by the proof of \cite[Proposition 6.8]{N2}, there exists $1 \le i \le m$ such that the situation of Lemma \ref{bad-case} occurs (for $(x, x', \a, r) = (x_{i-1}, x_i, \g_i, r_i)$).

Let $x, x', \a, r$ be as in Lemma \ref{bad-case}. If $\<\s^{r+d}(\a), \mu_x\> \le 0$, then $\<\s^r(\vartheta_\a), \mu_x\> \le -1$, which contradicts our assumption. So we have $\<\s^{r+d}(\a), \mu_x\> \ge 1$, and hence $$x \overset {(\s^{d+r}(\a), 3d-r)} \to y := x - \s^{r+d}(\a)^\vee + \s^d(\a)^\vee \overset {(\s^{d}(\a), r)} \to x.$$ Then it suffices to show that \begin{align*} \tag{a} g_2 I \sim_{\l, b} g_2 \o_2 I &\text{ for } g_2 \in \JJ_{b, \tw_y}; \\ \tag{b} g_1 I \sim_{\l, b} g_1 \o_1 I &\text{ for } g_1 \in \JJ_{b, \tw_x}, \end{align*} where $\o_1 = \s^{r+d}(\a)^\vee + \cdots + \s^{4d-1}(\a)^\vee, \o_2 = \s^d(\a)^\vee + \cdots + \s^{r+d-1}(\a)^\vee \in \pi_1(M_J) \cong \Omega_J$.

First we show (a). Note that $\<\s^r(\a), \mu_y\> = \<\s^r(\a), \mu_x\> \le -1$. We have $$y \overset {(\s^d(\a), r-d)} \to y - \s^d(\a)^\vee + \s^r(\a)^\vee \overset {(\s^r(\a), d)} \to x,$$ and (a) follows from Lemma \ref{small} and Corollary \ref{lin}.

Now we show (b). If $\<\a, \mu_x\> \le -1$, the statement follows from that $$x \overset {(\s^{r+d}(\a), 2d-r)} \to x - \s^{r+d}(\a) + \a^\vee \overset {(\a, d)} \to y.$$ So we can assume $\<\a, \mu_x\> = 0$. If $\<\s^i(\a), \mu_x\> = 0$ for $r+d+1 \le i \le 3d-1$, it follows from Lemma \ref{good-case}. Otherwise, let $k = \max\{r+d+1 \le i \le 3d-1; \<\s^i(\a), \mu_x\> \neq 0\}$. If $\<\s^k(\a), \mu_x\> = -1$, then $\<\s^{k-d}(\a), \mu_x\> \ge 1$ since $\<\s^k(\vartheta_\a), \mu_x\> \ge 0$, which means $x \overset {(\s^{k-d}(\d), 2d)} \to x_1 := x + \s^k(\a)^\vee - \s^{k-d}(\a)^\vee$ and $y \overset {(\s^{k-d}(\d), 2d)} \to y_1 := y + \s^k(\a)^\vee - \s^{k-d}(\a)^\vee$. By Lemma \ref{type-I}, we have $$g_1 \sim_{\l, b} g_1 \o' I \text{ for } g_1 \in \JJ_{b, \tw_x}, \ g_2 \sim_{\l, b} g_2 \o' I \text{ for } g_2 \in \JJ_{b, \tw_y},$$ where $\o' = \s^{k-d}(\d^\vee) + \cdots + \s^{k+d}(\d^\vee) \in \pi_1(M_J) \cong \Omega_J$. So we can replace the pair $(x, y)$ with $(x_1, y_1)$ so that $\<\s^k(\a), \mu_x\> \ge 1$. Then $$x \overset {(\s^k(\a), 4d-k)} \to x - \s^k(\a)^\vee + \s^d(\a)^\vee \overset {(\s^{r+d}(\a), k-r-d)} \to y,$$ and (b) follows from Lemma \ref{good-case}, Lemma \ref{small} and Corollary \ref{lin}.
\end{proof}

\begin{appendix}

\section{Distinct elements in $\Adm(\l)$} \label{sec-dist}
In this Appendix, we study the distinct elements introduced in \cite{CN1}.

\subsection{}
First we recall the following lemmas.
\begin{lem} \label{commute}
Let $s, s' \in \SS^a$ and $\tw \in \tW$ such that $\ell(s \tw)=\ell(\tw s')$ and $\ell(s \tw s')=\ell(\tw)$. Then $\tw = s \tw s'$.
\end{lem}

\begin{lem} [{\cite[Lemma 1.8 \& 1.9]{CN1}, \cite[Lemma 4.5]{Haines}}] \label{R1}
Let $s \in \SS^a$ and $\tw \in \Adm(\l)$ with $\l \in Y$ such that $\tw < s \tw$. Then we have

(1) $\tw s \in \Adm(\l)$ if $\tw s < s \tw s$;

(2) $\tw s = s \tw$ if $\tw s \notin \Adm(\l)$;

(3) $s \tw s \in \Adm(\l)$ if $\ell(s \tw s) = \ell(\tw)$.
\end{lem}

\begin{lem} \label{R4}
Let $\tw \notin \Adm(\l)$ and $s \in \SS^a$ such that $\tw s > \tw$. Then $s \tw s \notin \Adm(\l)$.
\end{lem}
\begin{proof}Assume $s \tw s \in \Adm(\l)$, then $s \tw s < \tw s$ and hence $\ell(s \tw s) = \ell(\tw)$. By Lemma \ref{R1} (3), we have $\tw \in \Adm(\l)$, contradicting the assumption that $\tw \notin \Adm(\l)$.
\end{proof}

\subsection{} Fix $\l \in Y^+$. Let $R \subseteq \SS_0$ and $\tw \in \Adm(\l)$. We say $\tw$ is left $R$-distinct (resp. right $R$-distinct) if $s \tw \notin \Adm(\l)$ (resp. $\tw s \notin \Adm(\l)$) for all $s \in R$. Let $w_R$ denote the longest element of $W_R$.

For a reflection $s \in W_0$ we denote by $\a_s \in \Phi^+$ the corresponding simple root.
\begin{lem} \label{R-dist}
Let $R = \{s, s'\} \subseteq \SS_0$. Let $\tw \in \Adm(\l)$ be right $R$-distinct. Let $u, u' \in W_R$ with $\ell(u') \le \ell(u)$. Then $u' \tw u\i \in\Adm(\l)$ if and only if $u = u'$. As a consequence, $w_R \tw w_R \in \Adm(\l)$ is left $R$-distinct.
\end{lem}
\begin{proof}
First we notice that $\tw \in \tW^R$, see \S \ref{setup}. Without loss of generality, we can assume $s \neq s'$ and $s s' s = s' s s'$.

First we show the ``only if'' part. By symmetry it suffices to consider the following cases.

Suppose $s \tw s' \in \Adm(\l)$. Then $s \tw s' < \tw s'$ and $s \tw s' s \notin \Adm(\l)$ (see Lemma \ref{R4}). By Lemma \ref{R1} we have $s \tw s' (\a_s) =\a_s$, that is, $\tw(\a_s + \a_{s'}) = -\a_s$. This is impossible since $\tw \in \tW^R$.

Suppose $s \tw s s' \in \Adm(\l)$. Then $s \tw s s' < s \tw s'$ (as $s \tw s' \notin \Adm(\l)$), that is, $s \tw s' (s'(\a_s)) = s \tw (\a_s) \in \tPhi^+$. Since $\tw(\a_s) \in \tPhi^-$ (as $\tw \in \tW^R$), we have $\tw(\a_s) = \a_s$. This means $s \tw s s' = \tw s' \notin \Adm(\l)$, a contradiction. Notice that $s \tw s' s \notin \Adm(\l)$ by Lemma \ref{R4}.

Suppose $s s' \tw s s' \in \Adm(\l)$. Then $s s' \tw s s' < s' \tw s s'$. If $s' \tw s s' <  s' \tw s s' s$, then $s s' \tw s s' s \notin \Adm(\l)$ by Lemma \ref{R4}. Otherwise, by Lemma \ref{commute} we have $s' \tw s s' s = \tw s s'$ (since $\tw s s' < \tw s s' s$) and hence $s s' \tw s s' s = s \tw s s' \notin \Adm(\l)$. So we always have $s s' \tw s s' s \notin \Adm(\l)$. By Lemma \ref{R1} we have $s s' \tw s s'(\a_s) = \a_s$, that is, $\tw(\a_{s'})= -(\a_s+\a_{s'})$, which is impossible as $\tw \in \tW^R$.

Suppose $s s' \tw s s' s \in \Adm(\l)$. Then $s s' \tw s s' s < s' \tw s s' s$. Since $s s' \tw s s' \notin \Adm(\l)$, by Lemma \ref{R1} we have $s s' \tw s s' s (\a_s) = \a_s$, that is, $\tw(\a_{s'}) = \a_s + \a_{s'}$. This means $s s' \tw s s' s = s' \tw s s' \in \Adm(\l)$, a contradiction.

Now we show the ``if'' part, that is, $u \tw u\i \in \Adm(\l)$ for $u \in W_R$. We argue by induction on the length $u$. If $u = 1$, the statement is true. Let $u = s u_1 > u_1$ with $u_1 \in W_R$ and $s \in R$. We assume $u_1 \tw u_1\i \in \Adm(\l)$ by induction hypothesis. It remains to show that $u \tw u\i \in \Adm(\l)$. Otherwise, we have $\ell(u \tw u\i)= \ell(u_1 \tw u_1\i) + 2$ and $u_1 \tw u\i \in \Adm(\l)$ by Lemma \ref{R1} (1) \& (3), which contradicts the ``only if'' part.
\end{proof}

\begin{lem} \label{LR}
Let $\tw \in \Adm(\l)$ and $s \in \SS_0$ such that $s \tw s \in \Adm(\l)$ and $s \tw \notin \Adm(\l)$. Let $\a\in \Phi^+ \setminus \{\a_s\}$ such that $\tw s_\a \in \Adm(\l)$. Then $s \tw s_\a s \in \Adm(\l)$
\end{lem}
\begin{proof}
Suppose $s \tw s_\a s \notin \Adm(\l)$, then $s \tw s_\a \in \Adm(\l)$ by Lemma \ref{R1}. As $s \tw \notin \Adm(\l)$, we have $s \tw(\a) \in \tPhi^+$. On the other hand, as $s(\a) \in \Phi^+$, $s \tw s_\a s \notin \Adm(\l)$ and $s \tw s \in \Adm(\l)$, we have $s \tw(\a) \in \tPhi^-$, which is a contradiction.
\end{proof}

\begin{cor} \label{conj}
Let $R = \{s, s'\} \subseteq \SS_0$. Let $\tw \in \Adm(\l)$ be left $R$-distinct. Let $\a \in \Phi^+ \setminus \Phi_R$ such that $\tw s_\a \in \Adm(\l)$. Then $u \tw s_\a u\i \in \Adm(\l)$ for $u \in W_R$.
\end{cor}
\begin{proof}
We argue by induction on $\ell(u)$,. If $u = 1$, the statement follows by assumption. Supposing it is true for $u_1$, that is, $ u_1 \tw u_1\i s_{u_1(\a)} = u_1 \tw s_\a u_1\i \in \Adm(\l)$, we show it is also true for $u = s u_1 > u_1$ with $s \in R$. By Lemma \ref{R-dist} we have $u_1 \tw u_1\i, s u_1 \tw u_1\i s \in \Adm(\l)$ and $s u_1 \tw u_1\i \notin \Adm(\l)$. Moreover, we have $u_1(\a) \neq \a_s$ since $\a \in \Phi^+ \setminus \Phi_R$. Thus $u \tw s_\a u\i = s u_1 \tw u_1\i s_{u_1(\a)} s \in \Adm(\l)$ by Lemma \ref{LR}.
\end{proof}

\end{appendix}

\end{document}